\documentclass{article}
\usepackage{amsfonts}
\usepackage{amsmath}
\usepackage{amsthm}
\usepackage{ascmac}
\usepackage{comment}
\usepackage[dvipdfmx]{graphicx}
\usepackage[dvipdfmx]{graphicx,color}
\usepackage{here}
\usepackage{relsize}
\usepackage{ulem}
\usepackage{color}
\usepackage{threeparttable}
\allowdisplaybreaks[1]

\theoremstyle{definition}
\newtheorem{theorem}{Theorem}
\newtheorem{definition}{Definition}
\newtheorem{proposition}{Proposition}
\newtheorem{lemma}{Lemma}

\newtheorem*{remark}{Remark}
\newtheorem{example}{example}
\newtheorem{assumption}{Assumption}

\newcommand{\hGamma}{\hat{\Gamma}}

\newcommand{\hG}{\hat{G}}
\newcommand{\ja}{\mathcal{J}^1}
\newcommand{\jb}{\mathcal{J}^0}
\newcommand{\jaa}{\mathcal{J}^{11}}
\newcommand{\jab}{\mathcal{J}^{10}}
\newcommand{\jba}{\mathcal{J}^{01}}
\newcommand{\jbb}{\mathcal{J}^{00}}
\newcommand{\tu}{\tilde{u}}
\newcommand{\hu}{\hat{u}}

\newcommand{\rti}{r_T^{-1}}
\newcommand{\rt}{r_T}
\newcommand{\ttheta}{\tilde{\theta}}
\newcommand{\htheta}{\hat{\theta}}
\newcommand{\dtheta}{\theta^{\dagger}}
\newcommand{\ctheta}{\check{\theta}}

\makeatletter
\@addtoreset{equation}{section}

\makeatother

\def\bd{\begin{description}}
\def\ed{\end{description}}

\def\D2{\bbD_{2,\infty-}}

\def\D{{\bf D}}

\def\I{{\bf I}}
\def\J{{\bf J}}

\def\calb{{\cal B}}

\def\calf{{\cal F}}

\def\cals{{\cal S}}

\def\sfd{{\sf d}}
\def\sfp{{\sf p}}

\def\up{\uparrow}

\def\be{\begin{equation}}
\def\ee{\end{equation}}
\def\bea{\begin{eqnarray}}
\def\eea{\end{eqnarray}}
\def\beas{\begin{eqnarray*}}
\def\eeas{\end{eqnarray*}}
\def\bi{\begin{itemize}}
\def\ei{\end{itemize}}
\def\im{\item}
\def\bd{\begin{description}}
\def\ed{\end{description}}
\newcommand{\bbB}{{\mathbb B}}

\newcommand{\bbD}{{\mathbb D}}

\newcommand{\bbR}{{\mathbb R}}

\newcommand{\bbY}{{\mathbb Y}}

\newif\ifrs
\rstrue
\ifrs \usepackage{mathrsfs} \fi  
\newif\ifcol
\coltrue 

\newtheorem{theorem*}{Theorem}[section]

\newtheorem{note*}[theorem*]{Note}
\newtheorem{lemma*}[theorem*]{Lemma}
\newtheorem{definition*}[theorem*]{Definition}
\newtheorem{proposition*}[theorem*]{Proposition}
\newtheorem{corollary*}[theorem*]{Corollary}
\newtheorem{remark*}[theorem*]{Remark}
\newtheorem{example*}[theorem*]{Example}
\numberwithin{equation}{section}
\newif\ifcol
\coltrue
\ifcol
\newcommand{\colorr}{\color[rgb]{0.8,0,0}}

\newcommand{\colorn}{\color[rgb]{1,1,1}}

\else

\newcommand{\colorr}{\color{black}}

\newcommand{\colorn}{\color{black}}

\fi
\excludecomment{en-text}
\includecomment{jp-text}

\usepackage{authblk}

\begin{document}
\title{Penalized least squares approximation methods and their applications to stochastic processes
\footnote{
This work was in part supported by 
Japan Science and Technology Agency CREST JPMJCR14D7; 
Japan Society for the Promotion of Science Grants-in-Aid for Scientific Research 
No. 17H01702 (Scientific Research);  
and by a Cooperative Research Program of the Institute of Statistical Mathematics. 
}
}
\author{Takumi Suzuki}
\author{Nakahiro Yoshida}
\affil{Graduate School of Mathematical Sciences, University of Tokyo
\footnote{Graduate School of Mathematical Sciences, University of Tokyo: 3-8-1 Komaba, Meguro-ku, Tokyo 153-8914, Japan. e-mail: nakahiro@ms.u-tokyo.ac.jp}
       }
\affil{CREST, Japan Science and Technology Agency
        }

\maketitle


\begin{abstract}
	We construct an objective function that consists of a quadratic approximation term and an $L^q$ penalty $(0<q\le1)$ term. Thanks to the quadratic approximation, we can deal with various kinds of loss functions into a unified way, and by taking advantage of the $L^q$ penalty term, we can simultaneously execute variable selection and parameter estimation. In this article, we show that our estimator has oracle properties, and even better property. We also treat an stochastic processes as applications.
\end{abstract}

\section{Introduction}
The least absolute shrinkage and selection operator (LASSO; Tibshirani 1996\nocite{tibshirani1996regression}) is a useful and widely studied approach to the problem of variable selection. Compared with other estimation methods, LASSO's major advantage is simultaneous execution of both parameter estimation and variable selection (\cite{fan2001variable}, \cite{tibshirani1996regression}).

Originally, LASSO was introduced for linear regression problems. Suppose that ${\bf y} = [y_1, ...,y_T]'$ is a response vector and ${\bf x}_j = [x_{1j}, ..., x_{Tj}]', j=1, ..., d,$ are the linearly independent predictors.
\footnote{ The prime denotes the matrix transpose. }
Then the LASSO estimator is defined by
\begin{align} \label{lasso}
	\htheta_{\rm LASSO} = \underset{\theta \in \mathbb{R}^d}{{\rm argmin}} \left\{ \Biggl\| {\bf y} - \sum_{j=1}^d {\bf x}_j \theta_j \Biggr\|^2 + \lambda \sum_{j=1}^d|\theta_j| \right\} ,
\end{align}
where $\lambda$ is a nonnegative regularization parameter. The second term in (\ref{lasso}) is the so-called $L^1$ penalty. Thanks to the singularity of the $L^1$ penalty at the origin, LASSO can perform automatic variable selection.

However, it is known that LASSO variable selection could be inconsistent (see e.g. \cite{zou2006adaptive}), because LASSO forces the coefficients to be equally penalized in $L^1$ penalty. Zou considers the different weights to different coefficients, and the estimator so obtained is called the adaptive LASSO estimator. More precisely, the adaptive LASSO estimator $\htheta_{\rm aLASSO}$ is defined by 
\begin{align*}
	\htheta_{\rm aLASSO} = \underset{\theta \in \mathbb{R}^d}{{\rm argmin}} \left\{ \Biggl\| {\bf y} - \sum_{j=1}^d {\bf x}_j \theta_j \Biggr\|^2 + \lambda_T \sum_{j=1}^d \hat{w}_j |\theta_j| \right\},
\end{align*}
where $\hat{w} = [\hat{w}_j]_j$ is a weight vector defined by $\hat{w}_j = 1/|\hat{\theta}_j|^{\gamma}$ for some constant $\gamma>0$ and an initial estimator $\htheta = [\hat{\theta}_j]_j$. The adaptive LASSO method requires consistency of $\htheta$, however thanks to different weights, variable selection is always consistent.

On the other hand, LASSO is easily extended to a general loss function $\mathcal{L}_T(\theta)$ as
\begin{align*}
	\htheta_{\rm LASSO} &= \underset{\theta \in \mathbb{R}^d}{{\rm argmin}} \left\{ \mathcal{L}_T(\theta) + \lambda \sum_{j=1}^d|\theta_j| \right\},
\end{align*}
and its adaptive version is given by
\begin{align*}
	\htheta_{\rm aLASSO} = \underset{\theta \in \mathbb{R}^d}{{\rm argmin}} \left\{ \mathcal{L}_T(\theta) + \lambda_T \sum_{j=1}^d \hat{w}_j |\theta_j| \right\}.
\end{align*}
Though this generalization enables us to apply LASSO type methods to various statistical models, asymptotic and numerical theories are established in a case-by-case manner. One of the solutions to this problem is the least squares approximation (LSA) method proposed by Wang and Leng (2007 \cite{wang2007unified}). LSA estimator is defined by
\begin{align*}
	\htheta_{\rm LSA} = \underset{\theta \in \mathbb{R}^d}{{\rm argmin}} \left\{ (\theta - \ttheta) \hGamma (\theta - \ttheta) + \lambda_T \sum_{j=1}^d \hat{w}_j |\theta_j| \right\},
\end{align*}
where $\hGamma$ is a non-singular matrix depending on the data. Using the LSA method for the adaptive LASSO, we can deal with many different models in a unified frame.

Choice of the penalty term is an crucial issue in regularization techniques. A popular method is the Bridge (\cite{frank1993statistical}) that uses an $L^q$ penalty term ($q>0$). Bridge estimation with $0<q<1$ has the "oracle properties" (\cite{knight2000asymptotics}). Oracle properties are proposed by Fan and Li (2001 \cite{fan2001variable}) and a good estimator with variable selection should have these properties. Let $\theta^* = [\theta_j^*]_j$ is the true value of $\theta$ and $\mathcal{A} = \{j; \theta_j^* \ne 0\}$. An estimator $\htheta$ has oracle properties if $\htheta$ satisfies
\begin{itemize}
	\item selection consistency: $P[ \htheta_{\mathcal{A}^c} =  0 ] \rightarrow 1$, and
	\item asymptotic normality: $\sqrt{T}(\htheta_{\mathcal{A}} - \theta_{\mathcal{A}}^*) \rightarrow^d N(0,\Gamma^{-1})$, for some $|\mathcal{A}| \times |\mathcal{A}|$ positive definite symmetric matrix $\Gamma$.
\end{itemize}

In this paper, we consider the objective function $Q_T^{(q)}(\theta)$ that consists of an LSA term and an $L^q$ penalty term:
\begin{align*}
	Q_T^{(q)}(\theta) = (\theta - \ttheta)' \hG (\theta - \ttheta) + \lambda_T \sum_{j=1}^d \hat{w}_j |\theta_j|^q.
\end{align*}
Then, for $0<q\le1$, we define penalized least squares approximation (penalized LSA) estimator by $\htheta^{(q)} = {\rm argmin}_{\theta}Q_T^{(q)}(\theta)$ and show that this estimator has oracle properties. Regarding variable selection, in particular, we give the convergence rate of the probability that variable selection succeeds correctly.

Applications to stochastic processes are also considered in this article. In particular, we are interested in point process and diffusion processes. For the point process, first, we consider the general theory of ergodic intensity model. Then, as an example, we treat the Cox process using the quasi likelihood analysis (QLA) method. For the diffusion process, we consider the ergodic and non-ergodic diffusion processes. We also use QLA method in this case.

This article is organized as follows. Section 2 introduces the model and the precise definition of penalized LSA estimators. Main results with respect to penalized LSA estimators are stated in Section 3. In Section 4, we also define the P-O estimator, which is multi-step estimator using LSA methods and has computational advantage. Proofs of theorems are given in Section 5. Finally, we study the application to stochastic processes in Sections 6 and 7 and report some simulations in Section 8.

\section{Definition of the penalized LSA estimator}
Suppose that $\theta = [\theta_1, ... ,\theta_{\sfp}]' \in \mathbb{R}^{\sfp}$ is a parameter of interest and $\ttheta = [\ttheta_1, ... ,\ttheta_{\sfp}]' \in \mathbb{R}^{\sfp}$ is an estimator of $\theta$. In many cases, $\ttheta$ minimizes some loss function $\mathcal{L}_T(\theta)$, but we will not assume the existence of the loss function. $\ttheta$ depends on $T$, however, we omit $T$ for the sake of simplicity : $\ttheta = \ttheta_T$.

\begin{example}
	Consider a linear regression model $y_t = {\bf x}_t' \theta + \epsilon_t, (t = 1,...,T, T\in\mathbb{N})$, where $\epsilon_t$ has a distribution with mean 0 and covariance $\sigma^2$ and $\{{\bf x}_t\}_t$ is independent of $\{\epsilon_t\}_t$. Then $\ttheta$ is the least square estimator for $\mathcal{L}_T(\theta) = \sum_t |y_t - {\bf x}'_t \theta|^2$.\end{example}

\begin{example}
	If we consider a negative log-likelihood function as a loss function, then $\ttheta$ is the maximum likelihood estimator (MLE) of $\theta$.
\end{example}

Hereafter, we assume that there exists a true value $\theta^* = [\theta^*_1, ... ,\theta^*_{\sfp}]' \in \mathbb{R}^{\sfp}$ of $\theta$ and that $\sfp^0$ components of $\theta^*$ do not equal to $0$, $\sfp^0 = \#\{ j ; \theta^*_j \ne 0 \}$. Here, for convenience of explanation, we consider a loss function $\mathcal{L}_T(\theta)$. In order to carry out parameter estimation and variable selection simultaneously, we consider adding a penalty term to the loss function $\mathcal{L}_T(\theta)$. For example, we can take a penalized loss function as the adaptive lasso objective function by Zou (2006 \cite{zou2006adaptive}):
\begin{align}
	\frac{1}{T} \mathcal{L}_T(\theta) + \sum_{j=1}^{\sfp} \kappa_T^j |\theta_j|, \label{adalasso}
\end{align}
where $\kappa_T^j = \alpha_T |\ttheta_j|^{-\gamma}$ for a deterministic sequence $(\alpha_T)_T$ and a $\sqrt{T}$-consistent estimator $\ttheta$. \par
We consider quadratic approximation of the loss function instead of the first term of (\ref{adalasso}). Thanks to this approximation, we can discuss the various cases into an unified methodology, and because the behavior at the infinity is simply described, we can argue more depth discussion like large deviation. Moreover, we replace $L^1$ penalty with $L^q$ penalty $(0<q\le1)$ instead. Under this setting, we will show that we can execute parameter estimation and variable selection simultaneously in this case. More precisely, for a $\sfp \times \sfp$ almost surely positive definite symmetric random matrix $\hG$ depending on $T$, we use the objective function
\begin{align*}
	Q_T^{(q)}(\theta) = \hG [(\theta-\ttheta)^{\otimes 2}]+ \sum_{j=1}^{\sfp} \kappa_T^j |\theta_j|^q,
\end{align*}
where $\kappa_T^j$ are nonnegative random variables, $A^{\otimes 2} = AA'$ for some matrix or vector $A$ and $A[B] = {\rm Tr}(AB')$ for matrices $A$ and $B$ of the same size. \par
For twice differentiable $\mathcal{L}_T(\theta)$, $T^{-1}\mathcal{L}_T(\theta)$ is approximated as
\[
	\frac{1}{T} \mathcal{L}_T(\theta) 
	\approx \frac{1}{T} \mathcal{L}_T(\ttheta) + \frac{1}{T} (\theta - \ttheta)' \partial_{\theta}{\mathcal{L}}_T(\ttheta) 
	+ \frac12 \Bigl\{ \frac{1}{T} \partial_{\theta}^2{\mathcal{L}}_T(\ttheta) \Bigr\} [(\theta - \ttheta)^{\otimes 2}].
\]
Here, the first term on the right hand side is constant with respect to $\theta$ and the second term vanishes by the definition of $\ttheta$. Thus, instead of minimizing $T^{-1} \mathcal{L}_T(\theta)$, we may minimize $\{T^{-1} \partial_{\theta}^2{\mathcal{L}}_T(\ttheta)\} [(\theta - \ttheta)^{\otimes 2}]$ and in this case we can take $\hG = T^{-1} \partial_{\theta}^2{\mathcal{L}}_T(\ttheta)$ for example. \par
Let $\htheta^{(q)} = [\htheta_1^{(q)}, ... ,\htheta_{\sfp}^{(q)}]'$ be a minimizer of this objective function $Q_T^{(q)}(\theta)$ :
\[
	\htheta^{(q)} \in \underset{\theta \in \Theta}{\text{argmin }}  Q_T^{(q)}(\theta)
\]
We call $\htheta^{(q)}$ the penalized least squares approximation (penalized LSA) estimator. 

\section{Results for penalized LSA estimator}
In this section, we will show asymptotic properties of the penalized LSA estimator $\htheta^{(q)}$ based on $Q_T^{(q)}(\theta)$. 
Suppose that the statistics are realized on a probability space $(\Omega,\calf,P)$.
To describe the results, we may suppose that $\theta^*_1 \ne 0, ... ,\theta^*_{\sfp^0} \ne 0$ and $\theta^*_{\sfp^0+1} = ... = \theta^*_{\sfp} = 0$ without loss of generality. 
Let
\[
	a_T = \max\{\kappa_T^j ; j \le \sfp^0 \} \quad \text{and} \quad b_T = \min\{\kappa_T^j ; j > \sfp^0 \}.
\]
For a vector $v = [v_1, ... ,v_{\sfp}]' \in \mathbb{R}^{\sfp}$, we denote subvectors $[v_1, ... ,v_{\sfp^0}]'$ and $[v_{\sfp^0+1}, ... ,v_{\sfp}]'$ by $v_{\ja}$ and $v_{\jb}$ respectively.
\par

We consider the following conditions with respect to $\ttheta$ and $\hG$. Let $r_T$ be a sequence of positive numbers tending to $0$ as $T \rightarrow \infty$. We often consider the case that $r_T = T^{-1/2}$.

\begin{assumption}\label{ass:cons_of_G}
	There exists a positive definite symmetric random matrix $G$ such that $\hG \rightarrow^p G$.
\end{assumption}
\begin{assumption}\label{ass:rootn}
	$\ttheta$ is $\rti$-consistent, i.e., $\rti(\ttheta - \theta^*) = O_p(1)$．
\end{assumption}
\begin{assumption}\label{ass:asynorm}
	$\rti(\ttheta - \theta^*) \rightarrow^{d_s} \Gamma^{-\frac12} \zeta$ holds, where $\Gamma$ is a $\sfp \times \sfp$ positive definite symmetric matrix, $\zeta$ is a $\sfp$-dimensional standard Gaussian random vector defined on an extended probability space of $(\Omega, \mathcal{F}, P)$ and independent of $\mathcal{G}$, and $d_s$ denotes the $\mathcal{G}$-stable convergence for some $\sigma$-field $\mathcal{G}$ such that $\sigma(\Gamma) \subset \mathcal{G} \subset \mathcal{F}$.
\end{assumption}

Of course, Assumption \ref{ass:asynorm} is stronger than Assumption \ref{ass:rootn}, but $\rti$-consistency and selection consistency of the penalized LSA estimator $\htheta^{(q)}$ are derived from Assumptions \ref{ass:cons_of_G} and \ref{ass:rootn}. We need Assumption \ref{ass:asynorm} to show asymptotic normality of penalized LSA estimator $\htheta^{(q)}$. 

For a $\sfp \times \sfp$ matrix $M = [m_{ij}]_{1\le i \le \sfp, 1\le j \le \sfp}$, we denote the $\sfp^0 \times \sfp^0$ matrix $[m_{ij}]_{1\le i \le \sfp^0, 1 \le j \le \sfp^0}$, $\sfp^0 \times (\sfp-\sfp^0)$ matrix $[m_{ij}]_{1 \le i \le \sfp^0, \sfp^0 < j \le \sfp}$, $(\sfp-\sfp^0) \times \sfp^0$ matrix $[m_{ij}]_{\sfp^0 < i \le \sfp, 1 \le j \le \sfp^0}$ and $(\sfp-\sfp^0) \times (\sfp-\sfp^0)$ matrix $[m_{ij}]_{\sfp^0 < i \le \sfp, \sfp^0 < j \le \sfp}$ by $M_{\jaa}, M_{\jab}, M_{\jba}$ and $M_{\jbb}$ respectively:
\begin{align*}
	M = \begin{bmatrix} M_{\jaa} & M_{\jab} \\ M_{\jba} & M_{\jbb} \end{bmatrix}.
\end{align*}

\begin{theorem}[$\rti$-consistency]\label{thm:rootn}
	Under Assumptions \ref{ass:cons_of_G} and \ref{ass:rootn},
	if $\rti a_T = O_p(1)$, 
	then 
	\[
		\rti(\htheta^{(q)} - \theta^*) = O_p(1).
	\]
\end{theorem}

\begin{theorem}[Selection consistency]\label{thm:select}
	Under Assumptions \ref{ass:cons_of_G} and \ref{ass:rootn},
	if $\rti a_T = O_p(1)$ and $\rt^{-(2-q)} b_T \rightarrow^p \infty$, 
	then 
	\[
		P[\htheta_{\jb}^{(q)} = 0] \rightarrow 1.
	\]
\end{theorem}

\begin{theorem}[Asymptotic normality]\label{thm:asynorm} 
	Let $\mathfrak{G} = \begin{bmatrix}  I_{\sfp^0} & (G_{\jaa})^{-1} G_{\jab} \end{bmatrix}$ for $\sfp^0 \times \sfp^0$ identity matrix  $I_{\sfp^0}$.
	Under Assumptions \ref{ass:cons_of_G} and \ref{ass:rootn}, if $\rti a_T = o_p(1)$ and $\rt^{-(2-q)} b_T \rightarrow^p \infty$, 
	then 
	\[
		\rti (\htheta^{(q)} - \theta^*)_{\ja} - \mathfrak{G} \{ \rti(\ttheta-\theta^*) \} \rightarrow^p 0.
	\]
	In particular, under Assumption \ref{ass:asynorm} and $G = \Gamma$, we have 
	\[
		\rti (\htheta^{(q)} - \theta^*)_{\ja} \rightarrow^{d_s} \mathfrak{G} \Gamma^{-\frac12}\zeta \sim {\rm MN}_{\sfp^0}(0, (\Gamma_{\jaa})^{-1}).
	\]
\end{theorem}
\vspace{10pt}
Hereafter, we consider $\kappa_T^j = \alpha_T |\ttheta_j|^{-\gamma}$, where $\gamma$ is a constant satisfying $\gamma > -(1-q)$ and $( \alpha_T )_T$ is a deterministic sequence. If $( \alpha_T )_T$ satisfies the conditions
\[
	\rt^{-(2-q+\gamma)} \alpha_T \rightarrow \infty
\]
and
\[
	\rti \alpha_T = o(1).
\]
Then Theorems \ref{thm:rootn}-\ref{thm:asynorm} follows from Assumptions \ref{ass:rootn} and \ref{ass:asynorm}. Moreover we will show that the probability $P[ \htheta^{(q)}_{\jb} = 0 ]$ can be evaluated by any power of $\rt$. \par
Let $\tu = \rti (\ttheta-\theta^*)$ and $\hu = \rti (\htheta^{(q)}-\theta^*)$. 
\begin{definition}
	For a stochastic process $X = \{ X_T \}_T$ is $L^{\infty-}$-bounded if and only if $\sup_T E[|X_T|^p] < \infty$ holds for all $p \ge 1$.
\end{definition}
Additionally, we consider the following conditions:

\begin{assumption} \label{ass:Lpbdd} \
	$\{ \hG \}_T$, $\{ \hG^{-1} \}_T$ and $\{ \tu \}_T$ are $L^{\infty-}$-bounded.
\end{assumption}
\begin{remark}
	The $L^p$-boundedness of a sequence of estimators can be obtained by the quasi likelihood analysis with a polynomial type large deviation inequality for an associated statistical random field. See \cite{yoshida2011polynomial} for details.
\end{remark}
\begin{theorem} \label{thm:rate}
	Let $\epsilon \in \bigl( -1+q, \gamma \bigr)$.
	We assume that $\rt^{1+\gamma-\epsilon} \alpha_T^{-1} = O(1)$ and $\rti \alpha_T = O(1)$.
	Then under Assumptions \ref{ass:cons_of_G} and \ref{ass:Lpbdd}, $\{ \hu \}_T$ is $L^{\infty-}$-bounded. Moreover, for all $L > 0$, there exists a constant $C_L$ such that 
	\begin{align*}
		P[\htheta_{\jb}^{(q)} = 0] \ge 1 - C_L \rt^{2L}
	\end{align*}
	for all $T>0$.
\end{theorem}

\section{P-O estimator}
We now consider the coefficient matrix $\hG$. In the above theorems, we assume convergence of $\hG$ to $G$ or $L^{\infty-}$-boundedness of $\{ \hG \}$ and $\{ \hG^{-1} \}$ but we should not necessarily find such coefficient matrix $\hG$. In fact, if we take $\hG = I_{\sfp}$, then we can prove Theorems 1-4 in the same way as Section 5 except that the conditional asymptotic variance in Theorem 3 becomes $(\Gamma^{-1})_{\jaa}$. Since $(\Gamma_{\jaa})^{-1} = (\Gamma^{-1})_{\jaa} - (\Gamma^{-1})_{\jab} ( (\Gamma^{-1})_{\jbb} )^{-1} (\Gamma^{-1})_{\jba}$, this estimator is not efficient. 
However, the objective function has following simple form
\[
	Q_T^{(q)}(\theta) = \sum_{j=1}^{\sfp}\Bigl( (\theta_j-\ttheta_j)^2 + \kappa_T^j |\theta_j|^q \Bigr).
\]
From a computational point of view, this fact is useful because it is difficult to optimize the non-convex function in the high-dimensional case. Then we calculate the new estimator under the model selected by the penalized LSA estimator with coefficient matrix $I_{\sfp}$. We call this estimator the P-O (penalized method to ordinary method) estimator and denote it by $\ctheta$. More precisely, we define the P-O estimator as follows. \par
Let $\Theta$ is a bounded open subset of $\mathbb{R}^{\sfp}$. First, we assume the $\rti$-consistency of the initial estimator $\ttheta$. Second, we get the penalized LSA estimator $\htheta^{(q)}_{I_{\sfp}}$ with coefficient matrix $I_{\sfp}$ defined by
\[
	\htheta^{(q)}_{I_{\sfp}} \in  \underset{\theta \in \Theta}{\text{argmin}} \sum_{j=1}^{\sfp}\Bigl( (\theta_j-\ttheta_j)^2 + \kappa_T^j |\theta_j|^q \Bigr),
\]
where $\kappa_T^j = \alpha_T |\ttheta_j|^{-\gamma}$.
Let $\hat{\jb} = \{ j=1,...,\sfp; \htheta^{(q)}_{I_{\sfp},j} = 0 \}$ and $\hat{\Theta} = \{ \theta \in \Theta; \theta_j = 0, j \in \hat{\jb} \}$. Here, we consider another loss function $\mathbb{L}_T(\theta)$. Then, we define the P-O estimator $\ctheta$ by
\[
	\ctheta \in \underset{\theta \in \hat{\Theta}}{\text{argmin }}  \mathbb{L}_T(\theta).
\]

Before we turn to the statement of results for the P-O estimator $\ctheta$, we consider some conditions. We denote a parameter $\theta = \begin{bmatrix} \phi \\ \psi \end{bmatrix} \in \mathbb{R}^{\sfp^0 + (\sfp-\sfp^0)}$ and its true value $\theta^* = \begin{bmatrix} \phi^* \\ \psi^* \end{bmatrix} = \begin{bmatrix} \phi^* \\ 0 \end{bmatrix}$. Let $\bar{\mathbb{L}}_T(\phi) = \mathbb{L}_T \Bigl( \begin{bmatrix} \phi \\ 0 \end{bmatrix} \Bigr)$ and $\bar{\phi} \in {\rm argmin}_{\phi} \bar{\mathbb{L}}_T(\phi)$.
\begin{assumption}\label{ass:3step}\
	\begin{enumerate}
	\renewcommand{\labelenumi}{(\roman{enumi})}
		\item $\{ \tu \}_T = \{ \rti(\ttheta-\theta^*) \}_T$ is $L^{\infty-}$-bounded.
		\item $\rti(\bar{\phi} - \phi^*) \rightarrow^{d_s} \Lambda^{-\frac12} \eta$, where $\Lambda$ is a $\sfp^0 \times \sfp^0$ positive definite symmetric random matrix, $\eta$ is a $\sfp^0$-dimensional standard Gaussian random vector independent of $\Lambda$.
		\item $\{ \rti(\bar{\phi} - \phi^*) \}_T$ is $L^{\infty-}$-bounded. \\
	\end{enumerate}
\end{assumption}

\begin{remark}
	In many cases, we take $\mathbb{L}_T(\theta) = \mathcal{L}_T(\theta)$ and $\Lambda=\Gamma_{\jaa}$. Then, we consider the sufficient condition for Assumption \ref{ass:3step}. We define the random field $\mathbb{Z}_T : \mathbb{U}_T \rightarrow \mathbb{R}_+$ by $\mathbb{Z}_T(u) = \exp \{ -\mathcal{L}_T(\theta^*+\rt u) + \mathcal{L}_T(\theta^*) \}$, where $\mathbb{U}_T = \{ u \in \mathbb{R}^{\sfp} ; \theta^* + \rt u \in \Theta \}$. We denote $B(R) = \{ u \in \mathbb{R}^{\sfp}; |u| \le R \}$. If $\mathbb{Z}_T(u) \rightarrow^{d_s} \mathbb{Z}(u)$ in $C(B(R))$ for every $R > 0$ as $T \rightarrow \infty$ and the initial estimator $\ttheta$ satisfies Assumption \ref{ass:rootn}, then  Assumption \ref{ass:3step} (ii) holds. Here, $\mathbb{Z}$ is a random field defined by $\mathbb{Z}(u) = \exp \Bigl( u' \Gamma^{\frac12} \zeta - \frac12 u' \Gamma u \Bigr)$.
	Moreover, if the random field $\mathbb{Z}_T$ satisfies polynomial type large deviation inequality (Theorem 1 in \cite{yoshida2011polynomial}), then Assumption \ref{ass:3step} (iii) holds.
\end{remark}

\begin{theorem}\label{thm:3step}
	\begin{enumerate}
	\renewcommand{\labelenumi}{(\alph{enumi})}
	\item Under Assumptions \ref{ass:asynorm}, \ref{ass:3step}(i) and (ii), we have
	\[
		\rti (\ctheta - \theta^*)_{\ja} - \rti(\bar{\phi}-\phi^*) \rightarrow^p 0.
	\]
In particular, we have
	\[
		\rti(\ctheta-\theta^*)_{\ja} \rightarrow^{d_s} \Lambda^{-\frac12} \eta \sim {\rm MN}_{\sfp^0}(0, \Lambda^{-1}).
	\]
	\item Let $\epsilon \in \bigl( -1+q, \gamma \bigr)$.
	We assume Assumptions \ref{ass:3step}(i) and (iii), $\rt^{1+\gamma-\epsilon} \alpha_T^{-1} = O(1)$ and $\rti \alpha_T = O(1)$.
	Then, we have $L^{\infty-}$-boundedness of $\{ \rti(\ctheta-\theta^*) \}_T$. Moreover, for all $L > 0$ there exists a constant $C_L$ such that 
	\begin{align} \label{cthetarate}
		P[\ctheta_{\jb} = 0] \ge 1 - C_L \rt^{2L}
	\end{align}
	for all $T>0$.
	\end{enumerate}
\end{theorem}

\section{Proofs}

\subsection*{Proof of Theorem \ref{thm:rootn}}
Since $\htheta^{(q)}$ minimizes $Q_T^{(q)}(\theta)$, we obtain
\begin{align}
	0 
	&\ge Q_T^{(q)}(\htheta^{(q)}) - Q_T^{(q)}(\theta^*) \notag \\ 
	&= \hG [(\htheta^{(q)} - \ttheta)^{\otimes 2}]
		+ \sum_{j=1}^{\sfp} \kappa_T^j |\htheta_j^{(q)}|^q 
		- \hG [(\theta^*  - \ttheta)^{\otimes 2}]
		- \sum_{j=1}^{\sfp} \kappa_T^j |\theta^*_j|^q \notag \\
	&= \hG [(\htheta^{(q)} - \theta^*)^{\otimes 2}]
		+ 2 (\htheta^{(q)} - \theta^*)' \hG (\theta^* - \ttheta)
		+ \sum_{j=1}^{\sfp} \kappa_T^j |\htheta_j^{(q)}|^q 
		-  \sum_{j=1}^{\sfp} \kappa_T^j |\theta^*_j|^q. \label{proof11}
\end{align}
Since $0 \le |\htheta^{(q)}_j| < |\theta^*|$ implies $(|\theta^*_j|^q - |\htheta^{(q)})j|^q)/(|\theta^*_j|-|\htheta^{(q)}_j|) \ge |\theta^*_j|^q/|\theta^*_j| = |\theta^*_j|^{q-1}$, we obtain $|\htheta^{(q)}_j|^q - |\theta^*_j|^q \ge -K^* |\htheta^{(q)}_j - \theta^*_j|$ where $K^* = \max_{1 \le j \le \sfp^0} |\theta^*_j|^{q-1}$. Thus 
\begin{align*}
	\sum_{j=1}^{\sfp} \kappa_T^j |\htheta_j^{(q)}|^q - \sum_{j=1}^{\sfp} \kappa_T^j |\theta^*_j|^q
	&\ge \sum_{j=1}^{\sfp^0} \kappa_T^j (|\htheta_j^{(q)}|^q - |\theta^*_j|^q) \\
	&\ge -\sum_{j=1}^{\sfp^0} K^* \kappa_T^j |\htheta_j^{(q)} - \theta^*_j| \\
	&\ge - \sfp^0 K^* a_T |\htheta^{(q)} - \theta^*|.
\end{align*}
Therefore, by multiplying both sides of (\ref{proof11}) by $\rt^{-2}$, we obtain
\begin{align*}
	0 &\ge \hG \bigl[ \{ \rti (\htheta^{(q)} - \theta^*) \}^{\otimes 2} \bigr] +2\{ \rti (\htheta^{(q)} - \theta^*) \}' \hG \{ \rti (\theta^* - \ttheta) \} - \sfp^0 K^* \rti a_T |\rti (\htheta^{(q)} - \theta^*)| \\
	&\ge \|\hG^{-1}\|^{-1} |\rti (\htheta^{(q)} - \theta^*)|^2 - 2\|\hG\| \cdot |\rti (\htheta^{(q)} - \theta^*)||\rti (\ttheta-\theta^*)| - \sfp^0 K^* \rti a_T |\rti (\htheta^{(q)} - \theta^*)|.
\end{align*}
After all,
\begin{align} \label{bdofht}
	|\rti (\htheta^{(q)} - \theta^*)|
	\le \Biggl\{ \|\hG^{-1}\| \Bigl( 2\|\hG\| \cdot |\rti (\ttheta-\theta^*)| + \sfp^0 K^* \rti a_T \Bigr) \Biggr\}.
\end{align}
Since the right hand side is $O_p(1)$ by the assumption, we obtain $\rti (\htheta^{(q)} - \theta^*) = O_p(1)$. \qed

\subsection*{Proof of Theorem \ref{thm:select}}
For some $j (\sfp^0 < j \le \sfp)$, we assume $\htheta_j^{(q)} \ne 0$. Since $Q_T^{(q)}(\theta)$ is differentiable at $\theta = \htheta^{(q)}$ with respect to the $j$-th component and $\htheta^{(q)}$ minimizes $Q_T^{(q)}(\theta)$, 
\[
	0 
	= \rti \frac{\partial{Q_T^{(q)}(\theta)}}{\partial{\theta_j}} |_{\theta=\htheta^{(q)}} 
	= 2 \hG^{(j)} \{ \rti (\htheta^{(q)} - \ttheta) \} + \rti \kappa_T^j q |\htheta_j^{(q)}|^{q-1} {\rm sgn} (\htheta_j^{(q)}),
\]
where $\hG^{(j)}$ means the $j$-th row vector of $\hG$. Therefore, we have 
\[
	2|\hG^{(j)} \{ \rti (\htheta^{(q)} - \ttheta) \} ||\rti \htheta_j^{(q)}|^{1-q} 
	= q \rt^{-(2-q)}  \kappa_T^j
	\ge q \rt^{-(2-q)} b_T.
\]
Since, by Theorem \ref{thm:rootn} and the assumption, the left hand side of above equation is $O_p(1)$ and $\rt^{-(2-q)} b_T \rightarrow^p \infty$, we obtain
\begin{align} \label{proofofthm2}
	P\Bigl[ \htheta_j^{(q)} \ne 0 \Bigr] 
	\le P\Bigl[ |2\hG^{(j)} \{ \rti (\htheta^{(q)} - \ttheta) \} ||\rti \htheta_j^{(q)}|^{1-q} \ge q \rt^{-(2-q)} b_T \Bigr] \rightarrow 0
\end{align}
for $j = \sfp^0+1,...,\sfp$.
\qed

\subsection*{Proof of Theorem \ref{thm:asynorm} }
For $\theta = \begin{bmatrix} \theta_{\ja} \\ \theta_{\jb} \end{bmatrix} \in \mathbb{R}^{\sfp}$,
\begin{align*}
	Q_T^{(q)}(\theta) 
	&= \hG [(\theta - \ttheta)^{\otimes 2}] + \sum_{j=1}^{\sfp} \kappa_T^j |\theta_j|^q \\
	&= \hG_{\jaa} [(\theta - \ttheta)_{\ja}^{\otimes 2}]
	+ 2(\theta - \ttheta)_{\ja}' \hG_{\jab} (\theta - \ttheta)_{\jb} + \hG_{\jbb} [(\theta - \ttheta)_{\jb}^{\otimes 2}] \\
	& ~~~~~~~~~~~~~~~~ + \sum_{j=1}^{\sfp^0} \kappa_T^j |\theta_j|^q + \sum_{j=\sfp^0+1}^{\sfp} \kappa_T^j |\theta_j|^q.
\end{align*}
In particular, for $\theta^{\ddagger} = \begin{bmatrix} \theta_{\ja} \\ 0 \end{bmatrix} \in \mathbb{R}^{\sfp}$,
\[
	Q_T^{(q)}(\theta^{\ddagger}) 
	= \hG_{\jaa} [(\theta - \ttheta)_{\ja}^{\otimes 2}] - 2(\theta - \ttheta)_{\ja}' \hG_{\jab} \ttheta_{\jb} + \hG_{\jbb} [\ttheta_{\jb}^{\otimes 2}] + \sum_{j=1}^{\sfp^0} \kappa_T^j |\theta_j|^q.
\]
Let
\begin{align*}
	A_T = \Bigr\{ \underset{1 \le j \le \sfp^0}{{\rm min}} |\htheta_j^{(q)}| > 0, \htheta_{\jb}^{(q)} = 0, \det(\hG_{\jaa}) \ne 0 \Bigr\}.
\end{align*}
Then Theorems \ref{thm:rootn} and \ref{thm:select} imply $P[A_T] \rightarrow 1$. Let $\mathbb{R}^{\sfp}_0 = \{ \theta \in \mathbb{R}^{\sfp}; \theta_{\jb} = 0 \}$. Since $Q_T^{(q)}(\htheta^{(q)}) = \underset{\theta^{\ddagger} \in \mathbb{R}^{\sfp}_0}{\rm min} Q_T^{(q)}(\theta^{\ddagger})$ on $A_T$, 
\begin{align*}
	0 
	&= \frac12 \frac{\partial{Q_T^{(q)}(\theta)}}{\partial{\theta_{\ja}}} \Bigl|_{\theta = \htheta^{(q)}} \\
	&= \hG_{\jaa}(\htheta^{(q)} - \ttheta)_{\ja} - \hG_{\jab} \ttheta_{\jb} + V(\htheta_{\ja}^{(q)})
\end{align*}
holds on $A_T$, where $V(\htheta_{\ja}^{(q)}) = \bigl[ 2^{-1} q\kappa_T^j |\htheta_j^{(q)}|^{q-1}{\rm sgn}(\htheta_j^{(q)}) \bigr]_{j=1,...,\sfp^0} \in \mathbb{R}^{\sfp^0}$. Let $\hat{\mathfrak{G}} = \begin{bmatrix} I_{\sfp^0} & (\hG_{\jaa})^{-1} \hG_{\jab} \end{bmatrix}$. Since $\hat{\mathfrak{G}} \rightarrow^p \mathfrak{G}$ and $1_{A_T} \{ \rti (\hG_{\jaa})^{-1} V(\htheta^{(q)}_{\ja}) \} \rightarrow^p 0$, we have 
\begin{align*}
	&\rti( \htheta^{(q)} - \theta^* )_{\ja} - \mathfrak{G} \{ \rti(\ttheta-\theta^*) \} \\
	&= 1_{A_T} \Bigl\{ \rti(\ttheta-\theta^*)_{\ja} + \rti(\hG_{\jaa})^{-1}\hG_{\jab}\ttheta_{\jb} - \rti(\hG_{\jaa})^{-1}V(\htheta^{(q)}_{\ja}) - \mathfrak{G} \{ \rti(\ttheta-\theta^*) \} \Bigr\} \\
	&\hspace{100pt} + 1_{A_T^c} \Bigl\{ \rti( \htheta^{(q)} - \theta^* )_{\ja} - \mathfrak{G} \{ \rti(\ttheta-\theta^*) \} \Bigr\} \\
	&= 1_{A_T} \Bigl\{ (\hat{\mathfrak{G}}-\mathfrak{G}) \{ \rti(\ttheta-\theta^*) \} - \rti(\hG_{\jaa})^{-1}V(\htheta^{(q)}_{\ja}) \Bigr\} + 1_{A_T^c} \Bigl\{ \rti( \htheta^{(q)} - \theta^* )_{\ja} - \mathfrak{G} \{ \rti(\ttheta-\theta^*) \} \Bigr\} \\
	&\rightarrow^p 0.
\end{align*}
\qed

\subsection*{Proof of Theorem \ref{thm:rate}}
By (\ref{bdofht}) and Assumption \ref{ass:Lpbdd}, $\{ \hu \}$ is $L^{\infty-}$-bounded. \par
For $j > \sfp^0$, by the inequality in (\ref{proofofthm2}) and the Markov's inequality, we have
\begin{align*}
	P\Bigl[ \htheta_j^{(q)} \ne 0 \Bigr]
	&\le P\Bigl[ 2|\hG^{(j)}|\cdot|\hu - \tu| \ge \rti \kappa_T^j q \bigl|\rt \hu \bigr|^{-(1-q)} \Bigr] \\
	&\le \frac{1}{\rt^{-(1-q+\epsilon)M}}2^M q^{-M} E\Biggl[ |\hG^{(j)}|^M|\hu - \tu|^M |\hu|^{M(1-q)} \Bigl( \frac{1}{\rt^{\epsilon-1} \kappa_T^j } \Bigr)^M \Biggr],
\end{align*}
where $M = M(L) = 2L(1-q+\epsilon)^{-1}$. Here, by H\"older inequality, we have
\begin{align} 
	&E\Biggl[ |\hG^{(j)}|^M|\hu - \tu|^M |\hu|^{M(1-q)} \Bigl( \frac{1}{\rt^{\epsilon-1} \kappa_T^j } \Bigr)^M \Biggr] \notag \\
	&\le E\Bigl[ |\hG^{(j)}|^{4M} \Bigr]^{\frac14} E\Bigl[ |\hu - \tu|^{4M} \Bigr]^{\frac14} E\Bigl[ |\hu|^{4M(1-q)} \Bigr]^{\frac14} E\Biggl[ \Bigl( \frac{1}{\rt^{\epsilon-1} \kappa_T^j } \Bigr)^{4M} \Biggr]^{\frac14}. \label{proof42}
\end{align}
Since
\begin{align*}
	E\Biggl[ \Bigl( \frac{1}{\rt^{\epsilon-1} \kappa_T^j } \Bigr)^{4M} \Biggr]
	&= E\Biggl[ \Bigl( \frac{|\ttheta_j|^{\gamma}}{\rt^{\epsilon-1} \alpha_T} \Bigr)^{4M} \Biggr] \\
	&\le \Biggl( \frac{1}{\rt^{-(1+\gamma-\epsilon)} \alpha_T} \Biggr)^M E\Bigl[ |\tu|^{4\gamma M} \Bigr]
\end{align*}
and $\{ \hu \}_T$,$\{ \tu \}_T$ and $\{ \hG \}$ are $L^{\infty-}$-bounded, the right-hand side of (\ref{proof42}) is bounded uniformly in $T$. This completes the proof.
\qed

\subsection*{Proof of Theorem \ref{thm:3step}}
Let $B_T = \{ \hat{\jb} = \{ \sfp^0+1,...,\sfp \} \}$. By Theorem \ref{thm:rate}, $P[B_T^c]$ is evaluated by any power of $\rt$. Therefore, (a) is obtained by
\begin{align*}
	|\rti(\ctheta-\theta^*)_{\ja} - \rti(\bar{\phi}-\phi^*)| \le 1_{B_T^c} \cdot 2\rti {\rm diam}(\Theta) \rightarrow^p 0,
\end{align*}
where ${\rm diam}(\Theta) = \sup\{ |\theta_1-\theta_2|; \theta_1,\theta_2 \in \Theta \}$. \\
(b) Similarly, since
\begin{align*}
	&\sup_T E[ | \rti(\ctheta-\theta^*) |^p ] \\
	&\le \sup_T E[ | \rti(\bar{\phi}-\phi^*) |^p ] + \sup_T \Bigl\{ P[B_T^c] \cdot (2\rti {\rm diam}(\Theta))^p \Bigr\} < \infty
\end{align*}
for all $p > 0$, we have $L^{\infty-}$-boundedness of $\{ \rti(\ctheta-\theta^*) \}_T$. By the definition of $\ctheta$, $\ctheta_{\jb} = 0$ is equivalent to $\htheta^{(q)}_{I_{\sfp},\jb} = 0$. Therefore we obtain the inequality (\ref{cthetarate}).
\qed

\section{Point processes}
\subsection{Ergodic intensity model}\label{290120-1}
In this section, we will apply the results in Section 3 to a point process with parameters containing zero components. We consider a multivariate point process $N=(N_t^\alpha)_{\alpha \in {\bf I}{, t\in\bbR_+}}$ with intensity process $\lambda(t,\theta)=(\lambda^{\alpha}(t,\theta))_{\alpha \in {\bf I}}$, $t\in{ \bbR_+}$, where ${\bf I} = \{1,2,...,\sfd\}$ is an index set. More precisely, given a stochastic basis $\mathcal{B} = (\Omega, \mathcal{F}, {\bf F}, P)$ with a filtration ${\bf F} = (\mathcal{F}_t)_{t \in \mathbb{R}_+}$, we suppose that $N$ and $\lambda(\cdot,\theta)$ are defined on $\calb$, the simple counting process $N$ is ${\bf F}$-adapted right-continuous, $\lambda(\cdot,\theta)$ is predictable locally integrable for every $\theta\in\Theta$, and that $N-\int_0^\cdot\lambda(s,\theta^*)ds$ is a $\sfd$-dimensional local martingale with respect to ${\bf F}$. Assume that the components of $N$ have no common jumps. The parameter space $\Theta$ is a bounded open set in $\bbR^\sfp$ that admits Sobolev's inequality 
\beas \label{sobolev}
\|f\|_\infty &\leq& C_\Theta\sum_{i=0,1}\|\partial_\theta^i f\|_{L^r(\Theta)}
\eeas
for elements $f$ of the Sobolev space $f\in W^{1,r}(\Theta)$, 
with a constant $C_\Theta$ independent of $f$, for $r>\sfp$. 
We suppose that $0\in\bbR^\sfp$ is in $\Theta$ and that 
the mapping 
$\theta\mapsto\lambda(t,\theta)$ is continuously extended to $\bar{\Theta}$. 

We will use the quasi likelihood method (\cite{clinet2017statistical})
with 
the quasi-log likelihood function
\begin{align} \label{qllfofpp}
	\ell_T(\theta) = \sum_{\alpha\in {\bf I}} \int_0^T \log{(\lambda^\alpha (t,\theta)} )dN^\alpha_t - \sum_{\alpha \in {\bf I}} \int_0^T \lambda^\alpha (t,\theta) dt.
\end{align}
Then $\mathcal{L}_T(\theta) = -\ell_T(\theta)$ becomes a loss function. 
The conditions stated later ensure the existence of the function (\ref{qllfofpp}).
For the initial estimator $\ttheta$, we can use, for example, 
the quasi maximum likelihood estimator $\ttheta^M$ and 
the quasi Bayesian estimator $\ttheta^B$ given by 
\beas 
\ttheta^M &\in& \underset{\theta\in\bar{\Theta}}{\text{argmax}} \>\ell_T(\theta)
\eeas
and 
\beas 
\ttheta^B &=& 
\left[\int_\Theta \exp(\ell_T(\theta))\pi(\theta)d\theta\right]^{-1}
\int_\Theta \theta\exp(\ell_T(\theta))\pi(\theta)d\theta,
\eeas
respectively, where $\pi$ is a prior density satisfying 
$0<\inf_\theta\pi(\theta)\leq\sup_\theta\pi(\theta)<\infty$. 

For ergodic point processes, asymptotic normality and convergence of moments of $\ttheta^M$ and $\ttheta^B$ were proved in \cite{clinet2017statistical}. 
We recall their results briefly. 
Hereafter $\theta^*\in\Theta$ denotes the true value of $\theta$ and 
the distribution of the data is expressed by a multivariate point process $N$ 
with intensity process $\lambda(t,\theta^*)$. 

\begin{assumption}\label{ass:b1}
The mapping $\lambda:\Omega\times\bbR_+\times\Theta\to\bbR_+^\sfd$ 
is $\calf\times\bbB(\bbR_+)\times\bbB(\Theta)$-measurable and almost surely satisfies 
\bd
\im[(i)] for every $\theta\in\Theta$, the mapping $s\mapsto\lambda(s,\theta)$ is left continuous,
\im[(ii)] for every $s\in\bbR_+$, the mapping $\theta\mapsto\lambda(s,\theta)$ 
is in $C^4(\Theta)$ and admits a continuous extension to $\bar{\Theta}$. 
\ed
\end{assumption}

\begin{assumption}\label{ass:b2}
{\bf (i)} $\displaystyle \sup_{t\in\bbR_+}\sum_{i=0}^4\big\|\sup_{\theta\in\Theta}
\partial_\theta^i\lambda(t,\theta)\big\|_p<\infty$ for every $p>1$. 
\bd
\im[(ii)] $\displaystyle \sup_{t\in\bbR_+}\big\|\sup_{\theta\in\Theta}
|\lambda^\alpha(t,\theta)^{-1}1_{\{\lambda^\alpha(t,\theta)\not=0\}}|
\big\|_p<\infty$ for $p>1$ and $\alpha\in\I$. 
\im[(iii)] For any $\theta\in\Theta$ and $\alpha\in\I$, 
$\lambda^\alpha(t,\theta)=0$ if and only $\lambda^\alpha(t,\theta^*)=0$. 
\ed
\end{assumption}

\begin{assumption}\label{ass:b3}
For every $(\alpha,\theta)\in\I\times\Theta$, there exists a probability measure 
$\nu^\alpha(\cdot,\theta)$ on $\bbR_+\times\bbR_+\times\bbR^\sfp$ and $0<\delta<\frac{1}{2}$ such that 
\beas 
\sup_{\theta\in\Theta}T^{\delta}\left\|
\frac{1}{T}\int_0^T f\big(\lambda^\alpha(t,\theta^*),\lambda^\alpha(t,\theta),
\partial_\theta\lambda^\alpha(t,\theta)\big)dt
-\int f(x,y,z) \nu^\alpha(dx,dy,dz,\theta)\right\|_p
&\to&0
\eeas
as $T\to\infty$ 
for $p>1$ and $f\in C_B(\bbR_+\times\bbR_+\times\bbR^\sfp)$. 
\end{assumption}

Let $\nu^\alpha(dx,dy,\theta)=\int_{\bbR^\sfp}\nu^\alpha(dx,dy,dz,\theta)$. 
Define $\mathbb{Y}_T(\theta)$ by
\[
	\mathbb{Y}_T(\theta) = \frac{1}{T}(\ell_T(\theta) - \ell_T(\theta^*)),
\]
and $\mathbb{Y}(\theta)$ by the limit in probability of $\mathbb{Y}_T(\theta)$, 
where 
\beas 
\bbY(\theta) &=& 
\sum_{\alpha\in\I}
\int_{\bbR_+\times\bbR_+}
1_{\{x,y>0\}}\big\{x\log(y/x)-(y-x)\big\}\nu^\alpha(dx,dy,\theta).
\eeas

\begin{remark}
	From the above expression of $\mathbb{Y}(\theta)$, we easily obtain $\mathbb{Y}(\theta^*)=0$ and for all $\theta \in \Theta$, 
	\begin{align} \label{negativey}
		\mathbb{Y}(\theta) \le 0.
	\end{align}
\end{remark}
Then Lemma 3.10 of \cite{clinet2017statistical} gives 
\beas 
\sup_{\theta\in\Theta}\big|\bbY_T(\theta)-\bbY(\theta)\big|
&\to^p&
0
\eeas
as $T\to\infty$. 

The index $\chi_0$ is defined by 
\beas 
\chi_0 &=& \inf_{\theta\in\Theta\setminus\{\theta^*\}}
\frac{-\bbY(\theta)}{|\theta-\theta^*|^2}. 
\eeas
Then identifiability is ensured by the condition 
\begin{assumption}\label{ass:b4}
$\chi_0>0$. 
\end{assumption}

The Fisher information matrix is well defined by 
\beas 
\Gamma &=& \sum_{\alpha\in\I} \int_{\bbR_+\times\bbR_+\times\bbR^\sfp}
z^{\otimes2}x^{-1}1_{\{x>0\}}\nu^\alpha(dx,dy,dz,\theta^*).
\eeas
The matrix $\Gamma$ is non-degenerate by Assumption \ref{ass:b4}. 

By Theorem 3.14 of \cite{clinet2017statistical}, we have 
\begin{theorem}
Suppose that Assumptions \ref{ass:b1}-\ref{ass:b4} are satisfied. 
Then 
for $\ttheta=\ttheta^M$ and $\ttheta^B$, the convergence 
\beas 
\lim_{T\to\infty}E\big[{\sf f}\big(\sqrt{T}(\ttheta-\theta^*)\big)]
&=& 
E\big[{\sf f}\big(\Gamma^{-1/2}\zeta\big)\big]
\eeas
holds for all $f\in C(\bbR^\sfp)$ of polynomial growth, where 
$\zeta$
is a $\sfp$-dimensional standard normal random variable. 
\end{theorem}

Now we are on the point of applying it to the penalized methods. 
Take $\ttheta=\ttheta^M$ or $\ttheta^B$. 
The penarized estimator will be denoted by $\hat{\theta}$. 
Let $\rt = T^{-\frac12}$ and let 
\beas 
\hG &=& 
-T^{-1}\partial_\theta^2\ell_T(\ttheta) 1_{\{-\partial_\theta^2\ell_T(\ttheta)\in\cals_+\}}
+T^{-1}I_{\sfp}
\eeas
where 
$\cals_+$ is the set of $\sfp\times\sfp$ positive definite symmetric matrices.
We embed the parametric model $\Theta$ into $\bbR^\sfp$ 
in use of the penalized method. 
It causes any problem asymptotically. 
If the reader prefers $\Theta$-valued estimators, he/she can use 
$\hat{\theta}1_{\{\hat{\theta}\in\Theta\}}+\theta_11_{\{\hat{\theta}\not\in\Theta\}}$ for $\hat{\theta}$ 
with a any given value $\theta_1\in\Theta$. 

It is easy to show 
\beas
\lim_{T\to\infty}\big\|T^{\delta} \big(\hG-\Gamma\big)\big\|_p
&=& 0
\eeas
for every $p>1$ and $0 < \delta < \frac12$. 
Therefore Assumptions \ref{ass:rootn}-\ref{ass:Lpbdd} are fulfilled and 
Theorems \ref{thm:rootn}-4 hold in this situation. 

\subsection{Cox type of process with ergodic covariates}\label{cox-ergodic}

We consider the multivariate point process $N$ in Section \ref{290120-1} 
with intensity processes
\begin{align}\label{290120-2}
	\lambda^\alpha (t,\theta) = \exp{\Bigl( \sum_{j \in {\bf J}} \theta_j^\alpha X_t^j \Bigr)},
	\text{\quad($\alpha \in {\bf I}$)}
\end{align}
where ${\bf J} = \{ 1,...,J \}$ is an index set and $X^j=(X_t^j)_{t\in\bbR_+}$ ($j \in {\bf J}$) are left-continuous adapted stochastic covariate processes satisfying the following conditions. 

\begin{assumption}\label{ass:c1}
The $J$-dimensional process $(X^j)_{j\in\J}$ is stationary and 
$E[\exp(uX^j_0)] <\infty$ for all $u \in \mathbb{R}$ and $j\in\J$. 
\end{assumption}

Denote by $\calb_I$ the $\sigma$-field generated by 
$\{X^j_t;\>t\in I,\>j\in\J\}$ for $I\subset\bbR_+$. 
Let 
\beas 
\alpha(h) &=& \sup_{A\in\calb_{[0,t]},B\in\calb_{[t+h,\infty)}}
\big|P[A\cap B]-P[A]P[B]\big|
\eeas
for $h>0$. 

\begin{assumption}\label{ass:c2}
There exists $a>0$ such that 
$\alpha(h)\leq a^{-1}e^{-ah}$ for $h>0$. 
\end{assumption}

Let $X_t=(X^j_t)_{j\in\J}$. 
For the model (\ref{290120-2}), $\theta=(\theta^\alpha_j)_{\alpha\in\I,j\in\J}$, $\sfp=\sfd J$ and 
\begin{align*}
	\hG
	&= \text{diag }[\hG_1,...,\hG_\sfd]
\end{align*}
where 
\begin{align}\label{290120-4}
	\hG_\alpha 
	&= \frac{1}{T}  \int_0^T X_t^{\otimes2} \exp \Bigl(\sum_{j \in {\bf J}} \ttheta_j^\alpha X^j_t \Bigr) dt+ \frac{1}{T}I_J.
\end{align}
It should be remarked that the first term on the right hand side of (\ref{290120-4}) 
may degenerate in general. 
Under Assumptions \ref{ass:c1} and \ref{ass:c2}, we obtain $\hG \to^p \Gamma$ for 
$\Gamma=\text{diag }[\Gamma_1(\theta^*),...,\Gamma_\sfd(\theta^*)]$, where 
\begin{align*}
	\Gamma_\alpha(\theta) &= 
	E\bigg[ X_0^{\otimes2} \exp \Bigl( \sum_{j \in {\bf J}} 
	\theta_j^\alpha X^j_0 \Bigr) \bigg]. 
\end{align*}
Write $\Gamma(\theta)=\text{diag }[\Gamma_1(\theta),...,\Gamma_\sfd(\theta)]$. 

\begin{assumption}\label{ass:c3}
$\inf_{\theta\in\bar{\Theta}}\det \Gamma(\theta)>0$. 
\end{assumption}

We assume that $\Theta$ is an open bounded convex subset in $\mathbb{R}^\sfp$ that admits 
the Sobolev inequality in Section \ref{290120-1}.

\begin{lemma}\label{lem:checkb3}
	Assumption \ref{ass:b3} holds under Assumptions \ref{ass:c1} and \ref{ass:c2}.
\end{lemma}
\begin{proof}[Proof] We remark that $\exp(|x|) < \exp(x) + \exp(-x)$ for all $x \in \mathbb{R}$. Thus, for all $(\theta, \alpha, j) \in \Theta\times{\bf I}\times{\bf J}$ and $p,q > 1$ and $t>0$,
\begin{align}
	E\Biggl[ |X_t^j|^p \Bigl\{ \exp( \theta_j^{\alpha} X_t^j ) \Bigr\}^q \Biggl]
	&\le E\Biggl[ \exp(p|X_t^j|) \exp( q|\theta_j^{\alpha}||X_t^j|) \Biggr] \notag \\
	&= E\Biggl[ \exp\Bigl\{ \big(p + q|\theta_j^{\alpha}|\big)|X_0^j|\Bigr\} \Biggr] < C_{p,q}, \label{lpofx}
\end{align}
where $C_{p,q}$ is a constant depend on $p, q$ but not depending of $\theta, i, j$. By the definition of $\lambda^{\alpha} (t,\theta)$, for all $\alpha \in {\bf I}$,
\begin{align*}
	\lambda^{\alpha}(t,\theta) &= \exp \Bigl(\sum_{j \in {\bf J}} \theta_j^\alpha X^j_t \Bigr), \\
	\partial_{\theta^{\alpha'}} \lambda^{\alpha} (t,\theta) &=
	\begin{cases}
		X_t \exp \Bigl(\sum_{j \in {\bf J}} \theta_j^\alpha X^j_t \Bigr) & \text{if } \alpha' = \alpha \\
		0 & \text{if } \alpha' \ne \alpha
	\end{cases},
\end{align*}
where $\theta^{\alpha} = [\theta_j^{\alpha}]_j$. 
For $f \in C_{\up}(\bbR_+\times\bbR_+\times\bbR^\sfp)$, $\alpha \in {\bf I}$ and $\theta \in \Theta$, define $\tilde{f}^{\alpha}_{\theta} \in D_{\up}(\mathbb{R}^J)$ by
\begin{align*}
	\tilde{f}^{\alpha}_{\theta}(x) = f({\rm e}^{\sum_j \theta_j^{*\alpha}x_j}, {\rm e}^{\sum_j \theta_j^{\alpha}x_j}, x{\rm e}^{\sum_j \theta_j^{\alpha}x_j}),
\end{align*}
where $D_{\up}(\mathbb{R}^J)$ is the set of continuous functions $\tilde{f}:x \rightarrow \tilde{f}(x)$ from $\mathbb{R}^J$ to $\mathbb{R}$ which are of polynomial growth in $(x,{\rm e}^{|x|})$.
Then we can write for all $\alpha \in {\bf I}$, 
\begin{align*}
	\tilde{f}_{\theta}^{\alpha}(X_t) = f\big(\lambda^\alpha(t,\theta^*),\lambda^\alpha(t,\theta), \partial_\theta\lambda^\alpha(t,\theta)\big).
\end{align*}
By (\ref{lpofx}), we obtain for all $\alpha \in {\bf I}$ and $p>1$, 
\begin{align*}
	\sup_{\theta \in \Theta} E[ |\tilde{f}_{\theta}^{\alpha}(X_t)|^p ] = \sup_{\theta \in \Theta} E[ |\tilde{f}_{\theta}^{\alpha}(X_0)|^p ] < \infty.
\end{align*}
Since
\begin{align*}
	\int f(x,y,z) \nu^\alpha(dx,dy,dz,\theta) = E[\tilde{f}_{\theta}^{\alpha}(X_0)]
\end{align*}
for all $\alpha \in {\bf I}$, we may show that for all $p>1$, $\alpha \in {\bf I}$ and $\tilde{f}_{\theta}^{\alpha} \in D_{\up}(\mathbb{R}^J)$,
\begin{align*}
	\sup_{\theta \in \Theta} T^{\frac12}\Biggl\| \frac{1}{T} \int_0^T \Bigl\{ \tilde{f}_{\theta}^{\alpha}(X_t) - E[\tilde{f}_{\theta}^{\alpha}(X_0)] \Bigr\}dt \Biggr\|_p = O(1).
\end{align*}
By Assumption \ref{ass:c1}, there exists a constant $C_0$ such that
\begin{align*}
	\Biggl\| \int_{s_1}^{s_2} \Bigl( \tilde{f}_{\theta}^{\alpha}(X_t) - E[\tilde{f}_{\theta}^{\alpha}(X_0)] \Bigr) dt \Biggr\|_p
		&\leq C_0(s_2-s_1)^p
\end{align*}
for $s_1 < s_2$. Then, Lemma 4 in \cite{yoshida2011polynomial} implies under Assumption \ref{ass:c2} that
\begin{align*}
	E\left[ \Biggl| \int_0^T \Bigl( \tilde{f}_{\theta}^{\alpha}(X_t) - E[\tilde{f}_{\theta}^{\alpha}(X_0)] \Bigr) dt \Biggr|^p \right]
	&= E\left[ \Biggl| \sum_{l=1}^{\lfloor T \rfloor} \int_{\frac{(l-1)T}{\lfloor T \rfloor}}^{\frac{lT}{\lfloor T \rfloor}} \Bigl( \tilde{f}_{\theta}^{\alpha}(X_t) - E[\tilde{f}_{\theta}^{\alpha}(X_0)] \Bigr) dt \Biggr|^p \right] \\
	&\le C_1\lfloor T \rfloor^{\frac{p}{2}} + C_2\lfloor T \rfloor \\
	&= O(T^{\frac{p}{2}}),
\end{align*}
for $T\ge1$ and $p\ge2$,
where $C_1$ and $C_2$ are constants depending on $a$ and $p$. Therefore, we have
\begin{align*}
	\sup_{\theta \in \Theta} T^{\frac12} \Biggl\| \frac{1}{T} \int_0^T \Bigl( \tilde{f}_{\theta}^{\alpha}(X_t) - E[\tilde{f}_{\theta}^{\alpha}(X_0)] \Bigr) dt \Biggr\|_p = O(1).
\end{align*}
\end{proof}

Next, we will give a sufficient condition for Assumption \ref{ass:b4}.
\begin{lemma}\label{lem:checkb4}
	We assume that $\Theta$ is convex.
	Then Assumption \ref{ass:b4} follows from Assumptions \ref{ass:c1} and \ref{ass:c3}.
\end{lemma}
\begin{proof}[Proof] By the definition of $\mathbb{Y}(\theta)$,
\begin{align*}
	\mathbb{Y}(\theta) 
	&= \sum_{\alpha \in {\bf I}} \bbY^{\alpha}(\theta),
\end{align*}
where $\mathbb{Y}^{\alpha}(\theta)$ is given by
\begin{align*}
	\mathbb{Y}^{\alpha}(\theta) = E\left[ \exp\Bigl(\sum_{j \in {\bf J}}\theta_j^{*\alpha} X_0^j\Bigr)\Bigl(\sum_{j \in {\bf J}}(\theta_j^{\alpha} - \theta_j^{*\alpha})X_0^j \Bigr) - \biggr\{ \exp\Bigl(\sum_{j \in {\bf J}} \theta_j^{\alpha}X_0^j \Bigr) - \exp\Bigl( \sum_{j \in {\bf J}} \theta_j^{*\alpha}X_0^j \Bigr) \Biggl\} \right]
\end{align*}
for $\alpha \in {\bf I}$. Thus we have
\begin{align*}
	\partial_{\theta}\mathbb{Y}(\theta) =
	\begin{bmatrix}
		\partial_{\theta^1}\mathbb{Y}^1(\theta) \\
		\partial_{\theta^2}\mathbb{Y}^2(\theta) \\
		\vdots \\
		\partial_{\theta^{\sfd}}\mathbb{Y}^{\sfd}(\theta) \\
	\end{bmatrix},
\end{align*}
where 
\begin{align*}
	\partial_{\theta^{\alpha}} \mathbb{Y}^{\alpha}(\theta)
	= E\left[ \Biggl\{ \exp\Bigl(\sum_{j \in {\bf J}}\theta_j^{*\alpha} X_0^j\Bigr) - \exp\Bigl(\sum_{j \in {\bf J}}\theta_j^{\alpha} X_0^j\Bigr) \Biggr\} X_0 \right].
\end{align*}
Similarly, 
\begin{align*}
	\partial_{\theta}^2 \mathbb{Y}(\theta) = {\rm diag}\Bigl[ \partial_{\theta^1}^2 \mathbb{Y}^1(\theta), \partial_{\theta^2}^2 \mathbb{Y}^2(\theta), ..., \partial_{\theta^{\sfd}}^2 \mathbb{Y}^{\sfd}(\theta) \Bigr],
\end{align*}
where 
\begin{align*}
	\partial_{\theta^{\alpha}}^2 \mathbb{Y}^{\alpha} (\theta) = - E\Bigl[ X_0^{\otimes2}\exp\Bigl(\sum_{j \in {\bf J}}\theta_j^{\alpha} X_0^j\Bigr) \Bigr] = - \Gamma_{\alpha}(\theta).
\end{align*}
Therefore, we have
\begin{align*}
	\partial_{\theta}^2 \mathbb{Y}(\theta) = -\Gamma(\theta).
\end{align*}
By Assumption \ref{ass:c3}, for all $\theta \in \bar{\Theta}$, $-\partial_{\theta}^2 \mathbb{Y}(\theta)$ is positive definite. Therefore, $-\mathbb{Y}(\theta)$ is a strictly convex function. We assume that there exists $\theta_1 \in \bar{\Theta}\setminus\{\theta^*\}$ such that $-\mathbb{Y}(\theta_1)=0$. Let $\phi(s)=(1-s)\theta^*+s\theta_1$ for $s\in[0,1]$. Then by convexity of $\bar{\Theta}$, $\phi(s)\in\bar{\Theta}$ for all $s\in[0,1]$, and by strict convexity of $-\mathbb{Y}(\theta)$, $-\mathbb{Y}(\phi(s)) < (1-s)(-\mathbb{Y}(\theta^*)) + s(-\mathbb{Y}(\theta_1)) = 0$, however since, from (\ref{negativey}), $-\mathbb{Y}(\cdot)$ is nonnegative, this is contradiction. Therefore, for all open neighborhoods $U(\theta^*)$ of $\theta^*$, 
\begin{align*}
	\inf_{\theta \in \bar{\Theta}\setminus U(\theta^*)} \frac{-\mathbb{Y}(\theta)}{|\theta-\theta^*|^2} > 0.
\end{align*}

Next, take a neighborhood $V(\theta^*)$ of $\theta^*$ such that for all $\theta \in V(\theta^*)$, there exists a vector $\dtheta$ such that 
\begin{align}
	-\mathbb{Y}(\theta) = -\frac12\partial_{\theta}^2\mathbb{Y} (\dtheta) [(\theta-\theta^*)^{\otimes 2}]. \label{out}
\end{align}
By positive definiteness of $-\partial_{\theta}^2\mathbb{Y}(\theta)$ and Assumption \ref{ass:c1},
\begin{align} 
	\inf_{\theta \in V(\theta^*)\setminus\{\theta^*\}} \frac{-\mathbb{Y}(\theta)}{|\theta-\theta^*|^2} 
	&\ge \inf_{\theta \in V(\theta^*)\setminus\{\theta^*\}} \frac{-\partial_{\theta}^2\mathbb{Y}(\dtheta)[(\theta-\theta^*)^{\otimes 2}]}{2|\theta-\theta^*|^2} \notag \\
	&\ge \inf_{\theta \in V(\theta^*)} \frac12 \tau_{\min}(\Gamma(\theta)) > 0. \label{in}
\end{align}
Therefore, by (\ref{out}) and (\ref{in}), we obtain $\chi_0 > 0$.
\end{proof}

Assumption \ref{ass:b1} follows from (\ref{290120-2}), and Assumption \ref{ass:b2} follows from Assumption \ref{ass:c1} and Sobolev's inequality. Thus, Theorems \ref{thm:rootn}-\ref{thm:rate} for the model (\ref{290120-2}) hold under Assumptions \ref{ass:b4}-\ref{ass:c2}.

\section{Diffusion type processes}
\subsection{Ergodic case}
Given a stochastic basis $(\Omega,\calf,{\bf F},P)$, $\mathbb{F}=(\calf_t)_{t\in\mathbb{R}_+}$, we consider a $d$-dimensional process $X = (X_t)_{t \in \mathbb{R}_+}$ 
adapted to the filtration ${\bf F} = (\mathcal{F}_t)_{t \in \mathbb{R}_+}$ and satisfying the following stochastic integral equation
\[
	X_t = X_0 + \int_0^t a(X_s, \theta_2) ds + \int_0^t b(X_s, \theta_1) dW_s, \quad t \in \mathbb{R}_+
\]
where $W$ is an $\sf r$-dimensional standard ${\bf F}$-Wiener process, $\theta = (\theta_1,\theta_2) \in \Theta_1 \times \Theta_2 = \Theta$ with $\Theta_1$ and $\Theta_2$ being bounded domains of $\mathbb{R}^{\sfp_1}$ and $\mathbb{R}^{\sfp_2}$, respectively, moreover $b : \mathbb{R}^d \times \Theta_1 \rightarrow \mathbb{R}^d \otimes \mathbb{R}^{\sf r}$ and $a : \mathbb{R}^d \times \Theta_2 \rightarrow \mathbb{R}^d$. We define the function $B$ by $B(x,\theta_1) = b(x,\theta_1)b(x,\theta_1)'$ and assume that $B(x,\theta_1)$ is invertible. We denote the true value of $\theta = (\theta_1,\theta_2)$ by $\theta^* = (\theta^*_1,\theta^*_2)$ and the number of active parameters of $\theta^*_k$ by $\sfp_k^0$ for $k=1,2$. We assume that each parameter space have a locally Lipschitz boundary. \par 
In this subsection, we assume that the process $X$ is ergodic. That is, there exists a unique invariant probability measure $\mu = \mu_{\theta^*}$ such that for any bounded measurable function $g : \mathbb{R}^d \rightarrow \mathbb{R}$, the convergence
\begin{align*}
	\frac{1}{T} \int_0^T g(X_t)dt \rightarrow^p \int_{\mathbb{R}^d} g(x) \mu(dx)
\end{align*}
holds. \par
We suppose that $0 \in \mathbb{R}^{\sfp_1 + \sfp_2}$ is in $\Theta$. Here we have the discrete-time observations $\mathbf{x}_n = (X_{t_i})_{i=0}^n$ and $\mathbf{y}_n = (Y_{t_i})_{i=0}^n$ where $t_i = ih$ with $h=h_n$ depending on $n$. We will consider the situation when $h_n \rightarrow 0$ and $nh_n^p \rightarrow 0$ as $n \rightarrow \infty$, and there exists $\epsilon_0 \in (0, \frac{p-1}{p})$ such that $n^{\epsilon_0} \le nh_n$ for large $n$. \par
Here, we assume the following properties of an initial estimator $\ttheta = (\ttheta_{1,n}, \ttheta_{2,n})$ :
\[
	(\sqrt{n}(\ttheta_{1,n} - \theta^*_1), \sqrt{nh}(\ttheta_{2,n} - \theta^*_2)) \rightarrow^d (\Gamma_1^{-\frac12}\zeta_1, \Gamma_2^{-\frac12}\zeta_2) \sim N_{\sfp_1 + \sfp_2}(0,{\rm diag}(\Gamma_1^{-1}, \Gamma_2^{-1}))
\]
and
\[
\sup_n\bigg(\big\|\sqrt{n}(\ttheta_{1,n} - \theta^*_1)\big\|_p+\big\|\sqrt{nh}(\ttheta_{2,n} - \theta^*_2)\big\|_p\bigg)
<\infty
\]
for every $p>1$,
where $\zeta_1$ and $\zeta_2$ are $\sfp_1$ and $\sfp_2$-dimensional standard normal variables respectively, and 
\begin{align*}
	\Gamma_1 &= \frac12 \int_{\mathbb{R}^d} {\rm Tr} \Bigl( B^{-1}(\partial_{\theta_1}B)B^{-1}(\partial_{\theta_1}B)(x,\theta^*_1) \Bigr) \mu(dx), \\
	\Gamma_2 &= \int_{\mathbb{R}^d}(\partial_{\theta_2}a(x,\theta^*_2)' B(x,\theta^*_1)^{-1} \partial_{\theta_2}a(x,\theta^*_2)) \mu(dx).
\end{align*}
We assume integrability and non-degeneracy of $\Gamma_1$ and $\Gamma_2$. It is known that the quasi maximum likelihood estimator, the quasi Bayesian estimator and the hybrid type estimators possess these properties under certain mild conditions (\cite{yoshida2011polynomial}, \cite{uchida2012adaptive}, \cite{uchida2014adaptive}, \cite{kamatani2015hybrid}).
For instance, if we use the hybrid multistep estimator $\ttheta^H = (\ttheta^H_{1,n}, \ttheta^H_{2,n})$ by Uchida and Kamatani (\cite{kamatani2015hybrid}) as an initial estimator $\ttheta$, then above conditions are satisfied by Theorem 1 of \cite{kamatani2015hybrid}. \par
For $q_1,q_2\in(0,1]$, we define the objective functions $Q_{1,n}^{(q_1)}$ and $Q_{2,n}^{(q_2)}$ by
\begin{align*}
	Q_{1,n}^{(q_1)} = \hG_{1,n}[(\theta_1-\ttheta_{1,n})^{\otimes 2}] + \sum_{i_1=1}^{\sfp_1} \kappa_{1,n}^{i_1} |\theta_1^{i_1}|^{q_1}
\end{align*}
and
\begin{align*}
	Q_{2,n}^{(q_2)} = \hG_{2,n}[(\theta_2-\ttheta_{2,n})^{\otimes 2}] + \sum_{i_2=1}^{\sfp_2} \kappa_{2,n}^{i_2} |\theta_2^{i_2}|^{q_2},
\end{align*}
respectively, where $\hG_{k,n} ~ (k = 1,2)$ are some $\sfp_k \times \sfp_k$ random matrices such that $\hG_{k,n} \to^p \Gamma_k$ and that the family $\big\{|\hG_{k,n}|+\big(\det\hG_{k,n}\big)^{-1}\big\}_{k,n}$ is $L^{\infty-}$-bounded, and $\kappa_{k,n}^{i_k} = \alpha_{k,n}|\ttheta_{k,n}^{i_k}|^{-\gamma_k}$ for some numbers $\gamma_k > -(1 - q_k)$ and some sequences $(\alpha_{1,n})_n$ and $(\alpha_{2,n})_n$ satisfying
\[
	(\sqrt{n})^{2-q_1+\gamma_1} \alpha_{1,n} \rightarrow \infty, (\sqrt{n}) \alpha_{1,n} \rightarrow 0
\]
and
\[
	(\sqrt{nh})^{2-q_2+\gamma_2} \alpha_{2,n} \rightarrow \infty, (\sqrt{nh}) \alpha_{2,n} \rightarrow 0
\]
respectively. Then we have the penalized LSA estimators $\htheta^{(q_1)}_{1,n}$ and $\htheta^{(q_2)}_{2,n}$ satisfying
\begin{align*}
	\htheta^{(q_1)}_{1,n} \in \underset{\theta_1 \in \bar{\Theta}_1}{\rm argmin} ~ Q_{1,n}^{(q_1)}(\theta_1)
\end{align*}
and
\begin{align*}
	\htheta^{(q_2)}_{2,n} \in \underset{\theta_2 \in \bar{\Theta}_2}{\rm argmin} ~ Q_{2,n}^{(q_2)}(\theta_2).
\end{align*}
For these estimators $\htheta^{(q_1)}_{1,n}$ and $\htheta^{(q_2)}_{2,n}$, Theorems \ref{thm:rootn}-\ref{thm:rate} hold respectively. Additionally, we consider the limit distribution of the joint variable $((\htheta^{(q_1)}_{1,n})_{\ja_1},(\htheta^{(q_2)}_{2,n})_{\ja_2})$. Here $\ja_k, \jaa_k$ and $\jab_k$ are defined  similarly to $\ja, \jaa$ and $\jab$, respectively, for each $k=1,2$. 

Now we can rephrase Theorems 1-5 in the present situation. In particular, 
\begin{proposition}
	The convergence
	\begin{align*}
		\Bigl( \sqrt{n}(\htheta^{(q_1)}_{1,n} - \theta^*_1)_{\ja_1}, \sqrt{nh}(\htheta^{(q_2)}_{2,n} - \theta^*_2)_{\ja_2} \Bigr)\rightarrow^d & \Bigl( \mathfrak{G}_1\Gamma_1^{-\frac12}\zeta_1, \mathfrak{G}_2 \Gamma_2^{-\frac12}\zeta_2 \Bigr) \\
		& \sim  N_{\sfp_1^0+\sfp_2^0} \Biggl( 0,{\rm diag}\Bigl( ((\Gamma_1)_{\jaa_1})^{-1},((\Gamma_2)_{\jaa_2})^{-1} \Bigr) \Biggr)
	\end{align*}
	holds, where $\mathfrak{G}_k = \begin{bmatrix}  I_{\sfp_k^0} & ((\Gamma_k)_{\jaa_k})^{-1} (\Gamma_k)_{\jab_k} \end{bmatrix}, k=1,2$.
\end{proposition}
\begin{proof}[Proof]
By Theorem \ref{thm:asynorm}, we have
\[
	\sqrt{n} (\htheta^{(q_1)}_{1,n} - \theta^*_1)_{\ja_1} - \mathfrak{G}_1 \bigl\{ \sqrt{n} (\ttheta_{1,n} - \theta^*_1)_{\ja_1} \bigr\} \rightarrow^p 0
\]
and
\[
	\sqrt{nh} ( \htheta^{(q_2)}_{2,n} - \theta^*_2 )_{\ja_2} - \mathfrak{G}_2 \bigl\{ \sqrt{nh} (\ttheta_{2,n} - \theta^*_2)_{\ja_2} \bigr\} \rightarrow^p 0.
\]
Therefore, 
\begin{align*}
	&\begin{bmatrix}
		\sqrt{n} (\htheta^{(q_1)}_{1,n} - \theta^*_1)_{\ja_1} \\
		\sqrt{nh} (\htheta^{(q_2)}_{2,n} - \theta^*_2)_{\ja_2}
	\end{bmatrix}
	\rightarrow^d
	\begin{bmatrix}
		\mathfrak{G}_1 \Gamma_1^{-\frac12} \zeta_1 \\
		\mathfrak{G}_2 \Gamma_2^{-\frac12} \zeta_2
	\end{bmatrix}	
	\\
	&\sim N_{\sfp_1^0+\sfp_2^0} \Biggl( 0,{\rm diag}\Bigl( ((\Gamma_1)_{\jaa_1})^{-1},((\Gamma_2)_{\jaa_2})^{-1} \Bigr) \Biggr)
\end{align*}
\end{proof}

\subsection{Non ergodic case : volatility estimation in finite time horizon}\label{nonergodic}
In this subsection, we will deal with the case where the Fisher information matrix is not deterministic.  We consider the following stochastic regression model
\begin{align}\label{modelofnonergodic}
	Y_t = Y_0 + \int_0^t b_s ds + \int_0^t \sigma(X_s,\theta) dW_s, \quad t \in [0,T] ,
\end{align}
where $W$ is an $\sf r$-dimensional standard Wiener process independent of the initial value of $Y_0$, $X$ and $b$ are progressively measurable processes with values in $\mathbb{R}^{\sf d}$ and $\mathbb{R}^{\sf m}$, respectively. $\sigma$ is an $\mathbb{R}^{\sf m} \otimes \mathbb{R}^{\sf r}$-valued measurable function defined on $\mathbb{R}^{\sf d} \times \Theta$, and $\Theta$ is a bounded domain in $\mathbb{R}^{\sfp}$ with a locally Lipschitz boundary. Additionally, we define $S = \sigma^{\otimes 2} = \sigma \sigma'$. The data set consists of discrete observations $(X_{t_j},Y_{t_j})_{j=0}^n$ with $t_j = jT/n$ and $T$ is fixed. \par
Here, we assume that there exists an estimator $\ttheta_n$ of $\theta^*$ such that
\[
	\sqrt{n} (\ttheta_n - \theta^*) \rightarrow^{d_s} \Gamma^{-\frac12} \zeta
\]
as $n\to\infty$, and for any continuous functions ${\sf f} : \mathbb{R}^{\sfp} \rightarrow \mathbb{R}$ of at most polynomial growth, 
\[
	E[ {\sf f} (\sqrt{n}(\ttheta_n - \theta^*)) ] \rightarrow E[ {\sf f}(\Gamma^{-\frac12}\zeta) ],
\]
where $\Gamma$ is the Fisher information matrix given by
\[
	\Gamma = \frac{1}{2T} \int_0^T {\rm Tr} \Bigl( (\partial_{\theta}S) S^{-1} (\partial_{\theta}S) S^{-1} (X_t, \theta^*) \Bigr) dt,
\]
$\zeta$ is a $\sfp$-dimensional standard normal random variable independent of $\Gamma$ and $\rightarrow^{d_s}$ means the $\sigma(\Gamma)$-stable convergence in distribution. Here we remark that the Fisher information matrix $\Gamma$ is not necessarily deterministic. In fact, Uchida and Yoshida \cite{uchida2013quasi} proved that the quasi maximum likelihood estimator and teh quasi Bayesian estimator have these properties under mild regularity conditions. An essential condition in their argument is the non-degeneracy of a key index $\chi_0$:
\begin{assumption}\label{20180926-1}
	For every $L > 0$, there exists $c_L > 0$ such that 
	\[
		P[\chi_0 \le r^{-1}] \le \frac{c_L}{r^L}\qquad (r>0)
	\]
	where
	\[
		\chi_0 = \inf_{\theta \ne \theta^*} \frac{-\mathbb{Y}(\theta)}{|\theta-\theta^*|^2}
	\]
	with
	\[
		\mathbb{Y}(\theta) = - \frac{1}{2T} \int_0^T \Biggl\{ \log\Bigl( \frac{\det S(X_t,\theta)}{\det S(X_t,\theta^*)} \Bigr) + {\rm Tr} \Bigl( S^{-1}(X_t,\theta)S(X_t,\theta^*) - I_d \Bigr) \Biggr\} dt.
	\]
\end{assumption}
For the initial estimator, for example, we can take the maximum likelihood type estimator $\ttheta^M_n$ that satisfies
\[
	\mathbb{H}_n(\ttheta^M_n) = \sup_{\theta \in \Theta} \mathbb{H}_n(\theta),
\]
or the Bayes type estimator $\ttheta^B_n$ for a prior density $\pi : \Theta \rightarrow \mathbb{R}_+$ with respect to the quadratic loss defined by
\[
	\ttheta^B_n = \Bigl( \int_{\Theta} \exp( \mathbb{H}_n(\theta) ) \pi(\theta) d\theta \Bigr)^{-1} \int_{\Theta} \theta \exp( \mathbb{H}_n(\theta) ) \pi(\theta) d\theta,
\]
where $\mathbb{H}_n(\theta)$ is a quasi-log likelihood function defined by
\begin{align} \label{qllfofndiff}
	\mathbb{H}_n(\theta) = -\frac12 \sum_{i=1}^n \Bigl\{ \log\det S(X_{t_{i-1}},\theta) + \frac{1}{h} S(X_{t_{i-1}},\theta)^{-1}[(\Delta_i Y)^{\otimes 2}] \Bigr\}.
\end{align}
Then we can use the QLA method in \cite{uchida2013quasi} to show the stable convergence and the $L^p$-boundedness of the estimators. In order to verify Assumption \ref{20180926-1} in practice, we may apply one of criteria given in \cite{uchida2013quasi}.
\begin{en-text}
\begin{remark*}
	About non-degeneracy of $\chi_0$, here we started with a good initial estimator. This means a certain good non-degeneracy has been assumed from beginning. Logically, we need to assume an ordinary non-degeneracy like [H2] in Uchida and Yoshida(2013 \cite{uchida2013quasi}). The geometric criterion is expected to work for multi-dimensional $\theta^*$.
\end{remark*}
\end{en-text}

Here we define the objective function
\[
	Q^{(q)}_n(\theta) = \hG_n [(\theta-\ttheta_n)^{\otimes 2}] + \sum_{i=1}^{\sfp} \kappa_n^i |\theta_i|^q
\]
and penalized LSA estimator $\htheta^{(q)}_n \in \underset{\theta \in \Theta}{\text{argmin}} ~ Q_n^{(q)}(\theta)$. We can take 
\bea\label{20180926-5} 
\hG_n = -\frac{1}{n} \partial_{\theta}^2 \mathbb{H}_n(\tilde{\theta}^M)
1_{\big\{-\partial_{\theta}^2 \mathbb{H}_n(\tilde{\theta}^M)\in\cals_+\big\}}+\frac{1}{n}I_\sfp
\eea 
\noindent
when we use the QMLE as an initial estimator. Let $\kappa_n^i = \alpha_n |\ttheta_{n,i}|^{-\gamma}$ for the number $\gamma > -(1-q)$ and the sequence $\alpha_n$ satisfying the conditions
\[
	(\sqrt{n})^{2-q+\gamma} \alpha_n \rightarrow \infty, \quad \sqrt{n} \alpha_n \rightarrow 0.
\]
Similarly to the previous sections, we can show that Theorems \ref{thm:rootn}-4 hold for this penalized LSA estimator. On the other hand, we should choose $\hG_n=I_\sfp$, in place of (\ref{20180926-5}), for the P-O estimator. 

\section{Simulations}
In this section we report two simulations. The first one is the Cox model in Section \ref{cox-ergodic} and the second one is the diffusion process in Section \ref{nonergodic}. For each simulation we perform 1000 Monte Carlo replications. \%( ) in Tables 1-4 denotes the number of times, in percentage over 1000 iterations, that the estimator chooses the true model. For comparison we use the unified LASSO type estimator and the Bridge type estimator. The unified LASSO type estimator is the special case of penalized LSA estimator where $q=1$. The Bridge estimator is not the special case of penalized LSA estimator, but we use this phraseology when the penalty has the form $\rti \sum_i |\theta_i|^q$, i.e., $r = 1, \gamma=0$ and $q<1$. \par
For the convenience of calculation, we use identity matrix as a coefficient matrix $\hat{G}$. Thus, the penalized LSA estimator is not efficient, and we use the P-O estimator in order to obtain the efficient estimator. \par
It has been shown through the simulation studies that the penalized LSA estimator can select the correct model if we choose appropriate tuning parameters and the P-O estimator has good performance for the active parameters.

\subsection{Simulation for the Cox model}
We consider the Cox model (\ref{290120-2}) in Section \ref{cox-ergodic} with $\alpha = 1$. Let $\sfp = 20$, then the parameter space $\Theta$ is $[ -10,10 ]^{20}$. The covariate process $X = (X_t)_{t \in [0, T]}$ is a 20-dimensional OU process satisfying the following stochastic differential equations
\[
	dX^i_t = -{\sf a}_i X^i_t dt + 0.4 dW^i_t, \quad X_0 = 0, \quad t \in [0,T]
\]
where ${\sf a}_i$ $( i=1,\cdots,20)$ are constants given by
\begin{alignat*}{5}
	{\sf a}_{1} = {\sf a}_{6}   = &{\sf a}_{11} = {\sf a}_{16} = 0.15, \\
	{\sf a}_{2} = {\sf a}_{7}   = &{\sf a}_{12} = {\sf a}_{17} = 0.2, \\
	{\sf a}_{3} = {\sf a}_{8}   = &{\sf a}_{13} = {\sf a}_{18} = 0.25, \\
	{\sf a}_{4} = {\sf a}_{9}   = &{\sf a}_{14} = {\sf a}_{19} = 0.3, \\
	{\sf a}_{5} = {\sf a}_{10} = &{\sf a}_{15} = {\sf a}_{20} = 0.35.
\end{alignat*}
and $W = (W^i)_{i = 1,\cdots,20}$ is a 20-dimensional standard Wiener process. Data $N = (N_t)_{t \in [0, T]}$ is a sample path of the point process with intensity $\lambda(t,\theta^*)$ in (\ref{290120-2}), where the true values $\theta^*$ of the parameter is
\[
	\theta^* = [2, -1, 1, -0.5, -1.5, 1.5, 0.5, 0.75, 0, 0, 0, 0, 0, 0, 0, 0, 0, 0, 0, 0]' .
\]
Let $\mathcal{L}_T(\theta) = \mathbb{L}_T(\theta) = -\ell_T(\theta)$ in (\ref{qllfofpp}) and we use QMLE for the initial estimator $\ttheta$. Then the objective function is denoted by
\[
	Q_T^{(q)}(\theta) = (\theta-\ttheta)' (\theta-\ttheta)+ \sum_{j=1}^{20} \kappa_T^j |\theta_j|^q.
\]
where $\kappa_T^j = \alpha_T |\ttheta_j|^{-\gamma}, \alpha_T = (\frac{1}{\sqrt{T}})^r$, $1 < r < 2-q+\gamma$. Let the triplet of tuning parameters $(\gamma, r, q) = (1, 1.2, 0.3)$. We will consider the cases $T=50, 100, 200$ and $400$.
Table 1 compares the results of the variable selection of the penalized LSA estimator, the unified LASSO type estimator and the Bridge type estimator. Here, the unified LASSO type estimator and the Bridge type estimator are the penalized LSA estimator with tuning parameter $(\gamma, r, q) = (1, 1.2, 1)$ and $(\gamma, r, q)=(0,1,0.3)$ respectively. \par
Table 2 compares the means and standard deviations (parentheses) for the three estimators (initial estimator, penalized LSA estimator and P-O estimator) and shows the results of the variable selection for penalized LSA estimator in the case $T=200$. \par
\begin{table}[htbp]
\caption{Results of the variable selection under T=50,100,200,400.}
\begin{center}
	\begin{tabular}{cccccc} \hline \hline
		 & $(\gamma,$ $r,$ $q)$ & $T=50$ & $T=100$ & $T=200$ & $T=400$ \\ \hline 
		\%(p-LSA) & (1, 1.2, 0.3) & 32.1 & 70.4 & 96.8 & 99.9 \\
		\% (unified LASSO) & (1, 1.2, 1) & 5.9 & 17.4 & 47.1 & 79.6 \\
		\% (Bridge type) & (0, 1, 0.3) & 8.9 & 21.9 & 52.3 & 80.3 \\ \hline
	\end{tabular}
\end{center}
\end{table}

\begin{table}[htbp]
\begin{center}
\caption{The summary of results for the simulation under $T=200$.}
	\scalebox{0.95}{
	\begin{tabular}{ccccccccccc} \hline \hline
	
		& $\theta_{1}$ & $\theta_{2}$ & $\theta_{3}$ & $\theta_{4}$ & $\theta_{5}$ & $\theta_{6}$ & $\theta_{7}$ & $\theta_{8}$ & $\theta_{9}$ & $\theta_{10}$ \\
		true & 2 & -1 & 1 & -0.5 & -1.5 & 1.5 & 0.5 & 0.75 & 0 & 0 \\ \hline
		
		initial & 1.9938 & -0.9936 & 0.9941 & -0.5003 & -1.4909 & 1.4978 & 0.5009 & 0.7490 & 0.0001 & -0.0002 \\
		 & \small(0.0722) & \small(0.0830) & \small(0.0858) & \small(0.0889) & \small(0.0962) & \small(0.0745) & \small(0.0851) & \small(0.0880) & \small(0.0908) & \small(0.0974) \\
		p-LSA & 1.9918 & -0.9872 & 0.9877 & -0.4758 & -1.4877 & 1.4946 & 0.4750 & 0.7383 & -0.0001 & -0.0008 \\
		& \small(0.0728) & \small(0.0840) & \small(0.0868) & \small(0.1035) & \small(0.0965) & \small(0.0748) & \small(0.1049) & \small(0.0904) & \small(0.0265) & \small(0.0340) \\
		P-O & 1.9959 & -0.9971 & 0.9964 & -0.4979 & -1.4931 & 1.4998 & 0.4946 & 0.7523 & -0.0003 & -0.0007 \\
		 & \small(0.0565) &  \small(0.0682) & \small(0.0713) & \small(0.0892) & \small(0.0823) & \small(0.0602) & \small(0.0915) & \small(0.0730) & \small(0.0228) & \small(0.0307) \\ 
		\%(p-LSA) & 100.0 & 100.0 & 100.0 & 98.8 & 100.0 & 100.0 & 98.4 & 100.0 & 99.4 & 99.2  \\ \hline \noalign{\vskip+.80mm} \cline{1-11}
		
		\multicolumn{1}{c}{} & $\theta_{11}$ & $\theta_{12}$ & $\theta_{13}$ & $\theta_{14}$ & $\theta_{15}$ & $\theta_{16}$ & $\theta_{17}$ & $\theta_{18}$ & $\theta_{19}$ & $\theta_{20}$ \\
		true & 0 & 0 & 0 & 0 & 0 & 0 & 0 & 0 & 0 & 0 \\ \hline
		initial & 0.0014 & -0.0038 & 0.0007 & 0.0026 & 0.0063 & -0.0015 & 0.0042 & 0.0013 & -0.0048 & -0.0014 \\
		 & \small(0.0760) & \small(0.0794) & \small(0.0872) & \small(0.0947) & \small(0.0971) & \small(0.0701) & \small(0.0802) & \small(0.0896) & \small(0.0914) & \small(0.0941) \\
		p-LSA & 0.0000 & 0.0003 & 0.0007 & 0.0002 & -0.0004 & 0.0000 & 0.0011 & -0.0003 & 0.0013 & 0.0004 \\
		 & \small(0.0000) & \small(0.0185) & \small(0.0252) & \small(0.0240) & \small(0.0248) & \small(0.0000) & \small(0.0194) & \small(00239) & \small(0.0247) & \small(0.0208) \\
		P-O & 0.0000 & -0.0003 & 0.0007 & 0.0000 & -0.0003 & 0.0000 & 0.0007 & -0.0004 & 0.0010 & 0.0004 \\
		 & \small(0.0000) & \small(0.0152) & \small(0.0221) & \small(0.0188) & \small(0.0259) & \small(0.0000) & \small(0.0134) & \small(0.0214) & \small(0.0199) & \small(0.0213) \\
		\%(p-LSA) & 100.0 & 99.7 & 99.4 & 99.4 & 99.5 & 100.0 & 99.7 & 99.5 & 99.7 & 99.5 \\ \cline{1-11}
	\end{tabular}
	}
\end{center}
\end{table}

\subsection{Simulation for a diffusion type process}
We consider the model (\ref{modelofnonergodic}) in Section \ref{nonergodic}. Let $\sfp = 10$, i.e., the parameter space $\Theta$ is $[-10,10]^{10}$. The process $Y$ is defined by 
\[
	Y_t = \int_0^t \sigma(X_s, \theta) dW_s, \quad t \in [0,1],
\]
where $X$ is a 10-dimensional OU process satisfying the following stochastic differential equation
\[
	dX_t = -0.2 X_t dt + 0.5 I_{10} dw_t, \quad X_0 = 0, \quad t \in [0,1]
\]
and $\sigma(x,\theta) = \exp(\sum_{j=1}^{10} \theta_j x_j) \wedge M_0, M_0 = 10^5$.
Here $w$ is a $10$-dimensional standard Wiener process independent of $W$. 
We generated
the data $(Y_{t_i}, X_{t_i})_{i = 0,1,...,n}, t_i = \frac{i}{n}$ with 
\[
	\theta^* = [1,1,-1,-1,0.5,0,0,0,0,0]' .
\]
Let $\mathcal{L}_n(\theta) = \mathbb{L}_n(\theta) = -\mathbb{H}_n(\theta)$ in (\ref{qllfofndiff}). We used the QMLE for the initial estimator $\ttheta$. \par
In order to apply our methods to this model, we use Theorem 6 in (\cite{uchida2013quasi}). By the definition of $\sigma(x,\theta)$, [A1] holds. [B2] is satisfied if we choose the stopping time $\tau \equiv 0$. Now we need to check [A3$'$]. Since ${\rm supp}\mathcal{L}\{X_0\} = \{0\}$, we can take $0 \in U \subset \{ x \in \mathbb{R}^{10} ; \sigma(x,\theta)<M,\forall \theta \}$. If we define $f(x,\theta) = (m_0\sum_j(\theta_j-\theta^*_j)x_j)/|\theta-\theta^*|$ for sufficiently small $m_0$ when $\theta \ne 0$ and $f(x,0) = \epsilon_0$ for some positive number $\epsilon_0$, then (i) is satisfied for $\varrho=2$. Next we take a covering $\{ \Theta_k \}_{k=1,...,11}$ such that $\Theta_k = \{ \theta \in \Theta; |\theta_k - \theta^*_k|>0, |\theta_k - \theta^*_k| \ge |\theta_j - \theta^*_j|, \forall j \}$ for $k=1,...,10$ and $\Theta_{11} = \{0\} \subset \Theta$. For $\Theta_k, (k=1,...,10)$, if we take $\xi_0 = e_k$ and $\Psi(P^{\perp}_{\xi_0}x, \theta) = (\sum_{j \ne k} (\theta_j-\theta^*_j)x_j)/(\theta_k-\theta^*_k)$, then $|f(x,\theta)| \ge \frac{1}{\sqrt{10}} (\xi_0 \cdot x + \Psi(P^{\perp}_{\xi_0}x, \theta) )$, and (ii) holds. \par
Then the objective function is denoted by
\[
	Q^{(q)}_n (\theta) = (\theta-\ttheta)' (\theta-\ttheta) + \sum_{j=1}^{10} \kappa_n^j |\theta_j|^q
\]
where $\kappa_n^j = \alpha_n |\ttheta_j|^{-\gamma}$, $\alpha_n = (\frac{1}{\sqrt{n}})^r, 1 < r < 2-q-\gamma$.
We considered the cases where $n=2500, 50000, 10000, 20000$ and the triplet of tuning parameters $(\gamma, r, q)=(3.2,1.2,0.3)$. In the same way as Table 1, Table 3 compares the results of the variable selection of the penalized LSA estimator, the unified LASSO type estimator and the Bridge type estimator. \par
Table 4 compares the means and the standard deviations (parentheses) for the three estimators (initial estimator, penalized LSA estimator and P-O estimator) in the case $n=10000$.
\\
\begin{table}[htbp]
\caption{Results for the variable selection under n=2500,5000,10000,20000.}
\begin{center}
	\begin{tabular}{cccccc} \hline \hline
		 & $(\gamma$, $r$, $q)$ & $n=2500$ & $n=5000$ & $n=10000$ & $n=20000$ \\ \hline
		\% (p-LSA) & (3.2, 1.2, 0.3) & 64.8 & 86.3 & 97.8 & 99.9 \\
		\% (unified LASSO) & (3.2, 1.2, 1) & 54.0 & 77.2 & 93.8 & 98.6 \\
		\% (Bridge type) & (0, 1, 0.3) & 0.7 & 0.9 & 1.0 & 1.7 \\ \hline
	\end{tabular}
\end{center}
\end{table}

\begin{table}[htbp]
\begin{center}
\caption{The summary of results for the simulation under $n=10000$.}
	\scalebox{0.95}{
	\begin{tabular}{ccccccccccc} \hline \hline
		 & $\theta_{1}$ & $\theta_{2}$ & $\theta_{3}$ & $\theta_{4}$ & $\theta_{5}$ & $\theta_{6}$ & $\theta_{7}$ & $\theta_{8}$ & $\theta_{9}$ & $\theta_{10}$ \\
		true & 1 & 1 & -1 & -1 & 0.5 & 0 & 0 & 0 & 0 & 0 \\ \hline
		initial & 1.0020 & 0.9952 & -0.9981 & -1.0017 & 0.5021 & -0.0040 & -0.0006 & -0.0027 & 0.0046 & 0.0019 \\
		 & \small(0.0879) & \small(0.0896) & \small(0.0835) & \small(0.0847) & \small(0.0842) & \small(0.0885) & \small(0.0896) & \small(0.0849) & \small(0.0839) & \small(0.0896) \\ 
		p-LSA & 1.0013 & 0.9945 & -0.9974 & -1.0010 & 0.4856 & -0.0003 & 0.0000 & 0.0000 & 0.0000 & 0.0003 \\
		 & \small(0.0881) & \small(0.0899) & \small(0.0838) & \small(0.0849) & \small(0.1097) & \small(0.0107) & \small(0.0000) & \small(0.0000) & \small(0.0000) & \small(0.0000) \\
		P-O & 1.0017 & 0.9980 & -0.9960 & -0.9978 & 0.4936 & -0.0002 & 0.0000 & 0.0000 & 0.0000 & 0.0002 \\
		 & \small(0.0663) & \small(0.0694) & \small(0.0731) & \small(0.0676) & \small(0.0916) & \small(0.0058) & \small(0.0000) & \small(0.0000) & \small(0.0000) & \small(0.0049) \\
		\%(p-LSA) & 100.0 & 100.0 & 100.0 & 100.0 & 98.0 & 99.9 & 100.0 & 100.0 & 100.0 & 99.9 \\ \hline
	\end{tabular}
	}
\end{center}
\end{table}

\def\cprime{$'$} \def\cprime{$'$}


\end{document}